\documentclass[oneside,11pt]{amsart}
\usepackage{amssymb, amscd, stmaryrd}
\usepackage[all]{xy}
\usepackage{enumerate}
\usepackage{geometry}
\usepackage{parskip}
\usepackage{hyperref}
\usepackage{amssymb}
\usepackage{amsmath}
\usepackage{amsthm}
\usepackage{mathtools}
\usepackage{tikz}
\usepackage{tikz-cd}
\usepackage{amsfonts}
\usepackage{blkarray}
\usepackage{xcolor}
\usetikzlibrary{matrix,arrows,shapes.geometric}

\newtheorem{thm}{Theorem}[section]
\newtheorem{defn}[thm]{Definition}
\newtheorem{prop}[thm]{Proposition}
\newtheorem{remk}[thm]{Remark}

\newtheorem{lem}[thm]{Lemma}
\newtheorem{cor}[thm]{Corollary}
\newtheorem{examp}[thm]{Example}
\newtheorem{ques}[thm]{Question}

\newcommand{\Z}{\mathbb Z}

\newcommand{\R}{\mathbb R}
\newcommand{\Q}{\mathbb Q}
\newcommand{\N}{\mathbb N}
\newcommand{\G}{\Gamma}

\newtheorem{thmx}{Theorem}

\title[Fibering flat manifolds and their fundamental groups]{Fibering flat manifolds of diagonal type and their fundamental groups}

\author{Ho Yiu Chung}
\address{School of Mathematics, University of Southampton, Southampton SO17~1BJ, UK}
\email{hyc1g16@southamptonalumni.ac.uk}

\author{Nansen Petrosyan}
\address{School of Mathematics, University of Southampton, Southampton SO17~1BJ, UK}
\email{N.Petrosyan@soton.ac.uk}


\date{\today}

\begin{document}
\begin{abstract} An $n$-dimensional closed flat manifold is said to be of diagonal type if the standard representation of its holonomy group $G$ is diagonal. An $n$-dimensional Bieberbach group of diagonal type is the fundamental group of such a manifold. We introduce  the {\it diagonal Vasquez invariant} of  $G$ as the least integer $n_d(G)$ such that every flat manifold of diagonal type with holonomy $G$  fibers over a flat manifold of dimension at most  $n_d(G)$ with flat torus fibers. Using a combinatorial description of Bieberbach groups of diagonal type, we give both upper and lower bounds for this invariant. We show that the lower bounds are exact when $G$ has low rank. We apply this to analyse diffuseness properties of Bieberbach groups of diagonal type. This leads to a complete classification of Bieberbach groups of diagonal type with Klein four-group holonomy and to an application to Kaplansky's Unit Conjecture.

\end{abstract}

\maketitle

\vspace{-3mm}

\section{Introduction} 
Closed flat Riemannian manifolds of diagonal type have received considerable attention in recent years due to their special properties. Subclasses of these manifolds include generalised Hantzsche-Wendt manifolds \cite{RS}, diagonal flat K\"{a}hler manifolds $M_{d\G}$ \cite{MP} and  real Bott manifolds.  For instance, Popko and Szczepa\'{n}ski in \cite{PS}  showed that Hantzsche-Wendt manifolds are cohomologically rigid. In \cite{Nan-LPPS}, it was shown that, in contrast to the case of real Bott manifolds, there are flat manifolds of diagonal type for which the existence of a spin structure may not be detected by its finite proper covers.
In \cite{MP}, Miatello and Podest\'{a}, constructed an infinite family of Dirac isospectral flat K\"{a}hler manifolds $M_{d\G}$ which are pairwise non-homeomorphic to each other, thus, giving another counterexample to Mark Kac's problem of ``Can one hear the shape of a drum?''  More recently,  in \cite{Gardam}, Gardam showed that the fundamental group of didicosm, which is a 3-dimensional non-diffuse Hantzsche-Wendt group, is a counterexample to the Kaplansky's Unit Conjecture.

An  $n$-dimensional {\it crystallographic group} $\Gamma$ is a discrete and  cocompact subgroup of the group of isometries $\mathrm{Isom}(\mathbb E^n)=O(n)\ltimes \R^n$ of the $n$-dimensional Euclidean space $\mathbb E^n$. A torsion-free crystallographic group is called  a {\it Bieberbach group}. 
It is well known that the fundamental group of any closed flat Riemannian manifold is a Bieberbach group and every such manifold arises as a quotient of an Euclidean space by an isometric action of a Bieberbach group (see \cite[Chapter II]{charlap}). 
By the First Bieberbach Theorem (see Section \ref{sec_bieb}), there is a short exact sequence
$$
0\to \Z^n\to \Gamma\to G\to 1
$$
where $\Z^n\cong \Gamma\cap \R^n$ is a maximal abelian subgroup of $\Gamma$ called the {\it lattice subgroup} and $G$ is a finite group called  the {\it holonomy group} of $\Gamma$. Given such a short exact sequence, it induces a representation $\rho:G\rightarrow \mathrm{GL}_n(\Z)$ called  the {\it holonomy representation} of $\Gamma$.  
It is well-known that $\rho$ is a faithful representation (see \cite[Chapter 2]{Szc cry}).

{A crystallographic group $\Gamma$ is said to be of {\it diagonal type} if it is isomorphic to a crystallographic group whose holonomy representation is diagonal.} 
As an immediate consequence of diagonality of the holonomy representation, it follows that the holonomy group of a crystallographic group of diagonal type is isomorphic to an elementary abelian $2$-group $C_2^k$ for some $k\geq 1$. A closed flat manifold is said to be of {\it diagonal type} if its fundamental group is of diagonal type.

In Section \ref{sec_bieb}, we  give a detailed introduction to crystallographic groups of diagonal type. In Section \ref{sec_diag}, we introduce the key combinatorial method for studying such groups.  Given a crystallographic group $\Gamma$ of diagonal type, we  associate to it a characteristic matrix $A_{\Gamma}$ with entries in the Klein four-group $\mathcal D$.  Studying certain additive operations over $\mathcal D$ in $A_{\Gamma}$, allows us to give a complete combinatorial characterisation of crystallographic groups  of diagonal type. This is a variation of the method introduced in \cite{PS} and in \cite{Nan-LPPS}.

In Section \ref{sec_vas}, we study the Vasquez invariant and introduce its analog for flat manifolds of diagonal type. Vasquez invariant allows one to determine whether a given  closed flat Riemannian manifold fibers over a lower dimensional flat Riemannian manifold with fibers flat tori \cite{Vas}. 

Let $M$ be a closed flat Riemannian manifold with the fundamental group $\pi_1(M)=\Gamma$. Let $T^k=\R^n/\Z^n$ be a flat torus where $\Gamma$ acts on it by isometries. Then $\Gamma$ also acts on the space $\widetilde{M}\times T^k$ by isometries, where $\widetilde{M}$ is the universal cover of $M$. It is easy to show that the space $(\widetilde{M}\times T^k)/\Gamma$ is a flat manifold 
 (see \cite[Section 2]{Vas}).  $(\widetilde{M}\times T^k)/\Gamma$ is called the {\it flat toral extension} of the manifolds $M$. We shall make the convention that a point is the 0-dimensional torus, and hence any flat manifold can be a flat toral extension of itself. 

We first recall the definition of the Vasquez invariant  introduced by A.~T.~Vasquez in \cite{Vas}.

\begin{thm}[Vasquez, {\cite[Theorem 3.6]{Vas}}]{\label{def vas alg}}
	  For any finite group $G$, there exists a number $n(G)\in\N$ minimal with the property that if $\Gamma$ is a Bieberbach group with holonomy group isomorphic to $G$, then the lattice subgroup $L\subseteq\Gamma$  contains a subgroup $N\lhd \G$ such that $\Gamma/N$ is a Bieberbach group of dimension at most $n(G)$.
\end{thm}

The  Vasquez invariant has an equivalent  geometric reformulation.

\begin{thm}[Vasquez, {\cite[Theorem 4.1]{Vas}}]\label{def vas geom} For any finite group $G$, there exists a number $n(G)\in\N$ minimal with the property that if $M$ is any compact flat Riemannian manifolds with holonomy group isomorphic to $G$, then $M$ is a flat toral extension of some compact flat Riemannian manifolds of dimension at most $n(G)$.
\end{thm}

The integer $n(G)$ is called the {\it Vasquez invariant or Vasquez number} of $G$.  In Section \ref{sec_vas}, we define the diagonal Vasquez invariant by modifying the definition of Vasquez invariant. 
\begin{defn}{\label{diag vas def}}
{\normalfont Let $G$ be an elementary abelian $2$-group. Define $n_d(G)$ to be the smallest integer with the property  that if $\Gamma$ is a Bieberbach group of diagonal type with holonomy group isomorphic to $G$, then its lattice subgroup   contains a normal subgroup $N$ such that $\Gamma/N$ is a Bieberbach group of diagonal type of dimension at most $n_d(G)$.

Equivalently, $n_d(G)$ can be defined as the smallest integer with the property  that any flat manifold of diagonal type with holonomy group $G$, fibers over a flat manifold of diagonal type of dimension at most $n_d(G)$ with flat torus fibers. We call $n_d(G)$ the {\it diagonal Vasquez invariant} of $G$. }
\end{defn}

Theorem \ref{thm-diag} ensures that $n_d(G)<\infty$. 

Let $R$ be a ring. Recall that a {\it characteristic algebra} of $M$ is the subalgebra of $H^*(M, R)$ generated by the characteristic classes, e.g.~Stiefel-Whitney classes, $R=\mathbb F_2$ or Pontryagin classes, $R=\Z$.  Analogously to \cite[Corollary 2.8]{Vas}, we obtain the following immediate application of the diagonal Vasquez invariant.

\begin{cor} Let $M$ be an  $n$-dimensional closed flat manifold of diagonal type with holonomy group $G$. Then the characteristic algebra of $M$ vanishes in dimensions greater than $n_d(G)$.
\end{cor}

In Section \ref{sec_vas}, we give a detailed introduction to the diagonal Vasquez invariant of finite groups. In Section \ref{sec_thm_a_b}, we estimate the diagonal Vasquez invariant by using the characteristic matrix  $A_{\Gamma}$ constructed in Section \ref{sec_diag} and obtain

\begin{thmx}{\label{main_bdd}}
	For $k\geq 2$, we have
	\begin{align*}
		5\cdot 2^{k-3}+1\geq n_d(C_2^k) &\geq
		\begin{cases}
		\frac{k(k+1)}{2}		& \text{if $k\geq 2$ is even},\\
		\frac{k(k+1)}{2}-1       & \text{if $k\geq 3$ is odd}.  
		\end{cases}
	\end{align*}
\end{thmx}

\begin{thmx}{\label{main_exact}}
	For $k\in\{1,2,3,4\}$, we have
		\begin{align*}
		n_d(C_2^k) &=
		\begin{cases}
		1		& \text{if $k=1$},\\
		3       & \text{if $k=2$}, 		\\
		5      	& \text{if $k=3$},		\\
		10		& \text{if $k=4$}.
		\end{cases}
		\end{align*}
\end{thmx}

By Remark \ref{remk_cd}, we know that $n_d(G)\leq n(G)$.  By \cite[Section 2]{CW},  $n(G)\leq \sum_{C\in\mathcal{X}}|G:C|$ where $\mathcal{X}$ is the set of conjugacy classes of subgroups of  $G$ of prime order. Moreover, $n(G)=\sum_{C\in\mathcal{X}}|G:C|$ if $G$ is a $p$-group (see \cite[Theorem 2]{CW}).  So for example,  $n(C_2^2)=6$ whereas by Theorem \ref{main_exact},  we have  $n_d(C_2^2)=3$. Thus,  any closed flat manifold with holonomy group $C_2^2$ of diagonal type is a flat toral extension of a compact flat $3$-manifold 
 and hence its characteristic algebra vanishes in dimensions greater than three.

\begin{ques}
	{\normalfont We can view the diagonal Vasquez invariant of $C_2^k$ as a function of the rank 
	$$f:\mathbb{N}\rightarrow\mathbb{N} \; : \; k\mapsto n_d(C_2^k).$$
	By Theorem \ref{main_bdd}, we know that $f$ is not linear. Is $f$  quadratic?}
\end{ques}

In Section \ref{sec_diffuse}, we turn our attention to the diffuseness property of Bieberbach groups. Diffuseness was introduced by B.~Bowditch in \cite{Bowditch}. It is an important property of a group which is related to the Kaplansky's Zero Divisor Conjecture and the connectivity property of  the group  $C^*$-algebra. The Zero Divisor Conjecture states that if a group $\Gamma$ is torsion-free and $R$ is an integral domain, then the group ring $RG$ has no zero divisors. B.~Bowditch discovered that the conjecture is true if the group $G$ is diffuse (see \cite[Proposition 1.1]{Bowditch}). 
For the definition of diffuseness we refer to Section \ref{sec_diffuse}. Instead, we state equivalent conditions for Bieberbach groups which is a combination of several results. 

\begin{thm}[{\cite{DW, HS, KR}}]\label{equiv_Bieb} Let $\Gamma$ be a Bieberbach group. Then the following are equivalent.
\begin{itemize}
\item[$(i)$] $\Gamma$ is diffuse.
\item[$(ii)$] The kernel of the trivial representation $C^*(G)\to  \mathbb C$ is a connective $C^*$-algebra.
\item[$(iii)$] $\Gamma$ is a poly-$\Z$ group.
\item[$(iv)$] Every nontrivial subgroup of $\Gamma$ has a nontrivial center.
\item[$(v)$] Every nontrivial subgroup of $\Gamma$ has a nontrivial first Betti number.
\end{itemize}
\end{thm}

Evidently, diffuse Bieberbach groups are well-understood. A group is said to be {\it non-diffuse} if it is not diffuse. We use the diagonal Vasquez invariant to characterise non-diffuse Bieberbach groups of diagonal type. The below two theorems are our main results in this direction.

\begin{thmx}{\label{main_c2xc2}}
	Let $\Gamma$ be an $n$-dimensional  Bieberbach group of diagonal type with holonomy group isomorphic to $C_2^2$. Then $\G$ is non-diffuse  if and only if 
	$$\Gamma\cong Z(\Gamma)\oplus(\Z^{n-k-3}\rtimes \Delta_P)$$
	where $k=b_1(\Gamma)$ and $\Delta_P$ is the 3-dimensional non-diffuse Hantzsche-Wendt group (also known as the Promislow group or Passman group).
\end{thmx}

 We immediately obtain the following complete characterisation. 

\begin{cor} \label{cor: complete} Let $\Gamma$ be an $n$-dimensional  Bieberbach group of diagonal type with holonomy group isomorphic to $C_2^2$. Then either $\Gamma$ is poly-$\Z$ or $\Gamma\cong Z(\Gamma)\oplus(\Z^{n-k-3}\rtimes \Delta_P)$.
\end{cor}

For general diagonal holonomy, we have

\begin{thmx}{\label{main_diffuse}}
	Let $\Gamma$ be a non-diffuse Bieberbach group of diagonal type. Then there exists $\Gamma'\leq\Gamma$ and a normal  poly-$\Z$ subgroup $N\unlhd\Gamma'$, such that $\Delta_P\cong\Gamma'/N$. In addition, if $\Gamma$ is a non-diffuse generalized Hantzsche-Wendt group, then $\Delta_P\leq\Gamma$.
\end{thmx}

$\Delta_P$ is the fundamental group of the flat $3$-manifold called didicosm by John Conway \cite{CR}. It is isomorphic to the Fibonacci group $F(2, 6)$. 
Recently, in \cite{Gardam}, Giles Gardam showed that $\Delta_P$ is a counterexample to the Kaplansky's Unit Conjecture; namely, the group ring of $\Delta_P$ over the field  of two elements $\mathbb F_2$ contains a nontrivial unit. Since the Unit Conjecture is known to hold for all poly-$\Z$ groups \cite{higman}, combining Gardam's counterexample with Corollary \ref{cor: complete} and Theorem \ref{main_diffuse}, one immediately obtains the following corollary.

\begin{cor}\label{Unit conj} Let $G$ be either a Hantzsche-Wendt group or a Bieberbach group of diagonal type with holonomy group isomorphic to $C_2^2$. Then $G$ satisfies the Unit Conjecture over the field $\mathbb F_2$ if and only if $G$ is poly-$\Z$.
\end{cor}


\begin{center}{\bf Acknowledgements}\end{center}

We would like to thank  Andrzej Szczepa\'{n}ski for helpful discussions and advice. We also thank the referee for the careful reading of the paper and all the recommendations which helped improve its content.

\section{Preliminaries on Bieberbach groups}{\label{sec_bieb}}

We begin this section by stating the first and third Bieberbach Theorems.

\begin{thm}[{First Bieberbach Theorem, \cite{Szc cry}}]
Let $\Gamma$ be an $n$-dimensional crystallographic group. Then $\Gamma\cap\R^n$ is a torsion-free and finitely generated abelian group of rank $n$, and is a maximal abelian, normal subgroup of finite index.
\end{thm}

\begin{thm}[{Third Bieberbach Theorem, \cite{Szc cry}}]
 Let $\Gamma_1$ and $\Gamma_2$ be $n$-dimensional crystallographic groups.  Then  $\Gamma_1$ is isomorphic to $\Gamma_2$ if and only if there exists $\alpha\in \mathrm{GL}_n(\R)\ltimes \R^n$ such that $\alpha\Gamma_1\alpha^{-1}=\Gamma_2$.
\end{thm}

\begin{remk}
{\normalfont Let $\gamma$ be an element in $\mathrm{Isom}(\mathbb E^n)\cong O(n)\ltimes \R^n$ defined by $x\mapsto Ax+a$ where $x\in \R^n$, $A\in O(n)$ and $a\in\R^n$. We  can view $\gamma$ either as a tuple $(A,a)$ or as an $(n+1) \times (n+1)$-matrix
	$\begin{pmatrix}
		A & a\\
		0 & 1
	\end{pmatrix}$.}
\end{remk}

Let $\Gamma$ be an $n$-dimensional crystallographic group of diagonal type. By The First Bieberbach Theorem, $\Gamma$ fits into the following  short exact sequence
	\begin{equation*}
	\begin{tikzpicture}[node distance=1.5cm, auto]
	\node (GA) {$\Gamma$};
	\node (G) [right of=GA] {$G$};
	\node (I) [right of=G] {$1$};
	\node (Z) [left of=GA] {$\mathbb{Z}^n$};
	\node (O) [left of=Z] {$0$};
	\draw[->] (GA) to node {$p$} (G);
	\draw[->] (G) to node {}(I);
	\draw[->] (O) to node {}(Z);
	\draw[->] (Z) to node {$\iota$}(GA);	 
	\end{tikzpicture}
	\end{equation*}
	where $G$ is a finite group, $\iota:\mathbb{Z}^n\hookrightarrow\Gamma$ is the inclusion map defined by $e_i\mapsto (I_n,e_i)$ where $e_1,...,e_n$ are the standard basis of $\Z^n$ and $p:\Gamma\rightarrow G$ is the projection map defined by $(A,a)\mapsto A$. Given such a short exact sequence, it induces a representation $\rho:G\rightarrow \mathrm{GL}_n(\mathbb{Z})$ given by $\rho(g)x=\bar{g}\iota(x)\bar{g}^{-1}$, where $x\in\mathbb{Z}^n$ and  $p(\bar{g})=g$. Thus, we can view $\Z^n$ as a $\Z G$-module. 
	
As an immediate consequence of diagonality of the holonomy representation, it follows that $G\cong C_2^k$ for some $k\geq 1$. Thus, we have
\begin{equation}{\label{eq_ses}}
	\begin{tikzpicture}[node distance=1.5cm, auto]
	\node (GA) {$\Gamma$};
	\node (G) [right of=GA] {$C_2^k$};
	\node (I) [right of=G] {$1$};
	\node (Z) [left of=GA] {$\mathbb{Z}^n$};
	\node (O) [left of=Z] {$0$};
	\draw[->] (GA) to node {$p$} (G);
	\draw[->] (G) to node {}(I);
	\draw[->] (O) to node {}(Z);
	\draw[->] (Z) to node {$\iota$}(GA);	 
	\end{tikzpicture}
	\end{equation}
Since $C_2^k$ is acting diagonally on $\Z^n$, invoking the third Bieberbach Theorem and by conjugating $\Gamma$ with a suitable element in $\mathrm{GL}_n(\R)\ltimes\R^n$, we can assume $\rho(g)$ is a diagonal matrix with all diagonal entries equal to $\pm 1$ for all $g\in C_2^k$. In other words, we can assume $\{e_1,...,e_n\}$ to be a set of basis of $\Z^n$ such that $g\cdot e_i=\pm e_i$ for all $i\in\{1,...,n\}$ and for all $g\in C_2^k$.  

Throughout the rest of this paper, if $\Gamma$ is a crystallographic group of diagonal type, we will assume the holonomy group acts diagonally on the lattice subgroup. 
We denote 
$$diag(a_1,...,a_n)=\begin{pmatrix}a_1 & & \text{\large 0}\\ & \ddots & \\ \text{\large 0} & & a_n\end{pmatrix}$$
Let $\omega\in H^2(C_2^k;\Z^n)$ be the cohomology class defining the short exact sequence (\ref{eq_ses}). Since $C_2^k$ acts diagonally on $\Z^n$, we have the isomorphism 
$$H^2(C_2^k;\Z^n)\cong H^2(C_2^k; M_1\oplus\cdots\oplus M_n)\cong H^2(C_2^k; M_1)\oplus\cdots\oplus H^2(C_2^k; M_n)$$
where $M_j\cong \Z$  and we can identify $\omega$ with $(\omega_1, \dots, \omega_n)$, where $\omega_j\in H^2(C_2^k; M_j)$ for $j=1,...,n$. The  action $C_2^k$ on $M_j$ extends to $\R$ and we obtain the following short exact sequence of $C_2^k$-modules
$$0\to M_j\to\R\to\R/{M_j}\to 0$$
By the corresponding long exact sequence in cohomology, we have 
\begin{equation}{\label{eq_21}}
H^2(C_2^k; M_j)\cong H^1(C_2^k; \R/M_j)
\end{equation}
 where under the isomorphism $\omega_j$ is identified with an element $$[\alpha_j]\in H^1(C_2^k; \R/M_j)\cong Der(C_2^k,\R/M_j)/P(C_2^k,\R/M_j)$$
where for each $j=1,...,n$,
\begin{equation}{\label{eq_der}}
Der(C_2^k,\R/M_j)=\{f:C_2^k\to \R/M_j\,|\, \forall x,y\in C_2^k,\,f(xy)=x\cdot f(y)+f(x)\}
\end{equation}
and
\begin{equation}
P(C_2^k,\R/M_j)=\{f:C_2^k\to \R/M_j\,|\,\exists m\in\R/\Z, \forall x\in C_2^k, f(x)=x\cdot m-m\}
\end{equation}

\begin{lem}{\label{lem_coho}}
Using the same notation as above, for any $j\in\{1,...,n\}$, there is a representative $\beta_j\in[\alpha_j]$ such that $\beta_j(g)\in\{[0],[\frac{1}{2}]\}$ for all $g\in C_2^k$.
\end{lem}
\begin{proof}

{Since $\alpha_j$ is a derivation, we have $\alpha_j(1)=0$, for all $j=1,...,n$.} First, we assume $C_2^k$ acts trivially on $\R/M_j$. It follows that $P(C_2^k,\R/M_j)$ is trivial. For any $g\in C_2^k$, by (\ref{eq_der}), we have
$$0=\alpha_j(1)=\alpha_j(gg)=g\cdot \alpha_j(g)+\alpha_j(g)=2\alpha_j(g)$$
{It follows that for all representatives $\beta_j\in[\alpha_j]$, we have $\beta_j(g)\in\{[0],[\frac{1}{2}]\}$ for all $g\in C_2^k$.} Next, we assume $C_2^k$ acts non trivially on $\R/M_j$. Let $g_1,...,g_k$ be generators of $C_2^k$ and assume without loss of generality that $g_1$ acts non-trivially on $\R/M_j$ and $g_i$ acts trivially on $\R/M_j$ for all $i=2,...,k$.
{Let $\beta_j\in Der(C_2^k,\R/M_j)$ be a derivation defined by  $\beta_j(g_1)=0$ and $\beta_j(g)=\alpha_j(g)$ for all $g\in \langle g_2,...,g_k\rangle$. We check that $\beta_j$ is indeed a derivation, since for all $g\in \langle g_2,...,g_k\rangle$, we have }
$$0=\beta_j(1)=\beta_j(gg)=g\cdot \beta_j(g)+\beta_j(g)=2\beta_j(g)$$ and
$$\beta_j(g_1g)=g_1\cdot \beta_j(g)+\beta_j(g_1)=-\beta_j(g)$$
Since $C_2^k=g\langle g_2,...,g_k\rangle\sqcup\langle g_2,...,g_k\rangle$, it follows that $\beta_j(g)\in\{[0],[\frac{1}{2}]\}$ for all $g\in C_2^k$. It remains to show that $\beta_j$ and $\alpha_j$ are in the same cohomology class. For all $g\in\langle g_2,...,g_k\rangle$, we have
$$\beta_j(g)=\alpha_j(g)$$
and
\begin{align*}
\beta_j(g_1g)-\alpha_j(g_1g) 	& =g_1\cdot \beta_j(g)+\beta_j(g_1)-g_1\cdot \alpha_j(g)-\alpha_j(g_1)\\
						& =-\beta_j(g)+\beta_j(g_1)+\alpha_j(g)-\alpha_j(g_1)\\
						& =-\alpha_j(g_1) 
\end{align*}
Thus we have
\begin{align*}
		(\beta_j-\alpha_j)(g)=
		\begin{cases}
		0				& \text{if $g$ acts trivially on $\R/M_j$},\\
		-\alpha_j(g_1)       	& \text{if $g$ acts non-trivially on $\R/M_j$}.
		\end{cases}
\end{align*}
Hence we have $(\beta_j-\alpha_j)(g)=g\cdot \left(\frac{\alpha_j(g_1)}{2}\right)-\frac{\alpha_j(g_1)}{2}$ for all $g\in C_2^k$. It follows that $\beta_j-\alpha_j\in P(C_2^k,\R/M_j)$, which finishes the claim.
\end{proof}

\begin{remk}{\label{remk_coho}} 
{\normalfont By the above lemma, we may assume $\alpha_j(g)\in\{[0],[\frac{1}{2}]\}$ for all $j\in\{1,...,n\}$ and for all $g\in C_2^k$  and thus $\alpha_j$ has order 1 or 2.  By a slight abuse of notation, we consider
${\alpha}_j(g)\in\{0, \frac{1}{2}\}$  under the identification of $\{[0],[\frac{1}{2}]\}$ with its lift $\{0, \frac{1}{2}\}\subseteq \Q$.}
\end{remk}
Let $q:\R^n\to \R^n/\Z^n$ be the natural projection. We say $s:C_2^k\to \R^n$ is a {\it vector system} for $\Gamma$ if $q\circ s:C_2^k\to \R^n/\Z^n$ is a derivation. For a vector system $s$ for $\G$, by \cite[Section 3]{Lut}, there are isomorphisms
$$
\Gamma\cong\left\{\begin{pmatrix}
\rho(g) & s(g)+z\\
0 & 1
\end{pmatrix}\Bigg\vert \, g\in C_2^k, \, z\in \Z^n\right\}
$$
and
\begin{equation}{\label{eq_cry_gen}}
\Gamma\cong\left\langle\begin{pmatrix}
\rho(g) & s(g)\\
0 & 1
\end{pmatrix}, \begin{pmatrix}
I_n & e_i\\
0 & 1
\end{pmatrix}\Bigg\vert \, g\in C_2^k, \, i\in\{1,...,n\}\right\rangle
\end{equation}
where $I_n$ is the $n$-dimensional identity matrix and $e_i$ is the $i^{th}$ column of $I_n$ and the image of ${\alpha}_j$ is $\{0,\frac{1}{2}\}$ for all $j=1,...,n$. Thus, we may take
 $${\alpha}=({\alpha}_1, \dots, {\alpha}_n):C_2^k\to\{0,\tfrac{1}{2}\}^n$$
 to be a vector system for $\Gamma$. So any $\gamma\in\Gamma$ can be expressed  as 
$$\gamma=(diag(a_a,...,a_n),(x_1,...,x_n))$$
 where $a_1,...,a_n\in \{-1,1\}$ and $x_1,...,x_n\in\{\frac{a}{2}\,|\,a\in\Z\}$. Moreover,  $\{\iota(e_1),...,\iota(e_n),\gamma_1,...,\gamma_k\}$ is a generating set for $\Gamma$ where $\gamma_i=\begin{pmatrix} \rho(g_i) & {\alpha}(g_i)\\ 0 & 1\end{pmatrix}$. We call $\{\gamma_1,...,\gamma_k\}$  the set of {\it non-lattice generators} for $\Gamma$.

\section{Combinatorics of Bieberbach groups of diagonal type}{\label{sec_diag}}

In this section, for each $n$-dimensional crystallographic group of diagonal type with holonomy group $C_2^k$, we define a $(2^k-1)\times n$-matrix which gives a complete combinatorial description of the crystallographic group of diagonal type.

Let $S^1$ be the unit circle in $\mathbb{C}$. We consider the elements $g_i\in \mathrm{Aut}(S^1)$ given by
$$g_0(z)=z,\,\,\, g_1(z)=-z,\,\,\,g_2(z)=\bar{z},\,\,\,g_3(z)=-\bar{z}$$
for all $z\in S^1$.

Equivalently, we can identity $S^1$ with $\R/\Z$. For any $[t]\in\R/\Z$, we have
$$g_0([t])=\left[t\right],\,\,\,g_1([t])=\left[t+\tfrac{1}{2}\right],\,\,\,g_2([t])=\left[-t\right],\,\,\,g_3([t])=\left[-t+\tfrac{1}{2}\right]$$
Let $\mathcal{D}=\langle g_i \,|\, i=0,1,2,3\rangle$. It is easy to check that 
\begin{equation}{\label{eq}}
g_3=g_1g_2,\,\, g_i^2=g_0 \,\,\text{and}\,\, g_ig_0=g_0g_i=g_i
\end{equation}
 for $i=1,2,3$. Notice that $\mathcal{D}$ is isomorphic to the Klein four-group. We define an action of $\mathcal{D}^n$ on $T^n$ by
$$(t_1,...,t_n)(z_1,...,z_n)=(t_1z_1,...,t_nz_n)$$
for $(t_1,...,t_n)\in \mathcal{D}^n$ and $(z_1,...,z_n)\in T^n=S^1\times\cdots\times S^1$. Any subgroup $\Z_2\subseteq \mathcal{D}^n$ defines a $1\times n$-row vector with entries in $\mathcal{D}$. We define a row vector with entries in the set $\{0,1,2,3\}$ under the identification $i\leftrightarrow g_i$ for $0\leq i \leq 3$.

Let $\Gamma$ be an $n$-dimensional crystallographic group of diagonal type and let $\omega\in H^2(C_2^k;\Z^n)$ be the cohomology class corresponding to standard extension of $\Gamma$. As mentioned at Section \ref{sec_bieb}, the class $\omega$ corresponds to
$$ [\alpha]=[(\alpha_1, \dots, \alpha_n)]\in H^1(C_2^k;\R/M_1)\oplus\cdots\oplus H^1(C_2^k;\R/M_n)$$ 
where $[\alpha]\in H^1(C_2^k;\R^n/\Z^n)$ and $M_j\cong \Z$ for $j=1,...,n$. Let $g\in C_2^k$ be a non-identity element and $\rho:C_2^k\to \Z^n$ be the holonomy representation of $\Gamma$. We have $\rho(g)=diag(X_1,...,X_n)$ and $\alpha(g)=(\alpha_1(g),...,\alpha_n(g))^T=(x_1,...,x_n)^T$ where $X_j\in\{1,-1\}$ and $x_j\in\left\{0,\frac{1}{2}\right\}$ for $j=1,...,n$. The corresponding element of $\mathcal D^n$ is an $n$-tuple $(t_1,...,t_n)\in\mathcal D^n$ defined by
$$t_j([t])=[X_jt+x_j]$$
where $t\in \R$ and $j\in\{1,...,n\}$. We define $A_\Gamma(g,M_j)=t_j\in\{0,1,2,3\}$ where $j\in\{1,...,n\}$ under the identification $i\leftrightarrow g_i$ for $0\leq i \leq 3$. In other words, we have
\begin{align}{\label{lab_action}}
A_\Gamma(g,M_j) &=
\begin{cases}
0       & \text{if $\rho(g)_{j,j}=1$ and $\alpha_j(g)=0$}, 		\\
1      	& \text{if $\rho(g)_{j,j}=1$ and $\alpha_j(g)=\frac{1}{2}$},		\\
2		& \text{if $\rho(g)_{j,j}=-1$ and $\alpha_j(g)=0$} ,\\
3		& \text{if $\rho(g)_{j,j}=-1$ and $\alpha_j(g)=\frac{1}{2}$}.
\end{cases}
\end{align}
Denote by $h_1,...,h_{2^k-1}$ all the non-identity elements of $C_2^k$. We define a $(2^k-1)\times n$-dimensional matrix $A_\Gamma$ as $(A_\Gamma)_{i,j}=A_\Gamma(h_i,M_j)$.

Note that given $\Gamma$, the matrix $A_{\Gamma}$ is not unique since we could re-index the holonomy group elements $h_i$ and the modules $M_i$. We say that $A_{\Gamma}$ is a {\it characteristic matrix} of $\Gamma$. 

Let $r_1=(a_1 \cdots a_n)$ and $r_2=(b_1 \cdots b_n)$  be two rows of $A_\Gamma$. We denote by $\star$ the ``sum'' corresponding to the multiplication (\ref{eq}). In other words, we define $a \star b =c$ if $g_{a}g_{b}=g_{c}$ and define $r_1\star r_2=(a_1\star b_1\,\, \cdots \,\,a_n\star b_n)$.

\begin{lem}\label{lem_sum}  Let $\Gamma$ be an $n$-dimensional Bieberbach group of diagonal type. Let 
$$[\alpha]=[(\alpha_1, \dots, \alpha_n)] \in H^1(C_2^k;\R/M_1)\oplus\cdots\oplus H^1(C_2^k;\R/M_n)$$
where $M_j\cong \Z$ for all $j=1,...,n$ be the cohomology class corresponding to standard extension of  $\Gamma$ and let $h_1,h_2,h_3\in C_2^k$. If $h_1=h_2h_3$, then for any $j\in\{1,...,n\}$, we have 
$$A_\Gamma(h_1,M_j)=A_\Gamma(h_2,M_j)\star A_\Gamma(h_3,M_j)$$
\end{lem}
\begin{proof}
Let $\rho:C_2^k\to GL_n(\Z)$ be the holonomy representation of $\Gamma$. Let $\rho(h_2)=diag(X_1,...,X_n)$, $\alpha(h_2)=(x_1,...,x_n)$, $\rho(h_3)=diag(Y_1,...,Y_n)$ and $\alpha(h_3)=(y_1,...,y_n)$. We have
$$\begin{pmatrix} X_1 & & \text{\large 0} & x_1\\
&\ddots & &\vdots\\
\text{\large 0}& & X_n & x_n\\
0& \cdots & 0 & 1\\
\end{pmatrix}
\begin{pmatrix} Y_1 & & \text{\large 0} & y_1\\
&\ddots & &\vdots\\
\text{\large 0}& & Y_n & y_n\\
0& \cdots & 0 & 1\\
\end{pmatrix}=
\begin{pmatrix} X_1Y_1 & & \text{\large 0} & X_1y_1+x_1\\
&\ddots & &\vdots\\
\text{\large 0} & & X_nY_n & X_ny_n+x_n\\
0&\cdots  &0 & 1\\
\end{pmatrix}$$

Fix $j\in\{1,...,n\}$, the entry $A_\Gamma(h_1,M_j)$ is defined by
$$g[t]=[X_jY_jt+X_jy_j+x_j]=[X_j(Y_jt+y_j)+x_j]$$
where $t\in\R$ and $g\in\mathcal D$. Therefore, we have 
$$A_\Gamma(h_1,M_j)=A_\Gamma(h_2,M_j)\star A_\Gamma(h_3,M_j).$$
\end{proof}

\begin{remk}{\label{remk_gen}}
	{\normalfont Using the same notation as above, assume $C_2^k$ is generated by the elements $h_1,...,h_k$. There is a matrix $A$  obtained from $A_{\G}$ by deleting some of its rows. More specifically, define a $k\times n$-submatrix $A$ of $A_{\G}$ by $A_{i,j}=A_\Gamma(h_i,M_j)$ where $1\leq i\leq k$ and $1\leq j\leq n$.  By Lemma \ref{lem_sum},  the matrix $A_{\G}$ is the matrix generated by the rows of $A$ using the $\star$ operation. The matrix $A$ is the same as the matrix constructed in \cite[Section 2]{Nan-LPPS}. }
\end{remk}

Next, we would like to reverse the above construction. Namely, given a $k\times n$-dimensional matrix $A$ with entries in $\{0,1,2,3\}$, we are going to define an $n$-dimensional crystallographic group $\G_{\widetilde{A}}$ of diagonal type.

Let $\widetilde{A}$  be a $(2^k-1)\times n$-dimensional characteristic matrix  generated by rows of $A$. Denote by $g_1,...,g_{2^k-1}$  all non-identity elements of $C_2^k$ and
assume $g_1,...,g_k$ are the generators of $C_2^k$. First, we need to define a representation $\rho:C_2^k\to \mathrm{GL}_n(\Z)$. For any $1\leq i\leq k$, we define $\rho(g_i)=diag(X_1,...,X_n)$ where
\begin{align*}
X_j &=
\begin{cases}
1       & \text{if $(\widetilde{A})_{ij}\in\{0,1\}$} , 		\\
-1      	& \text{if $(\widetilde{A})_{ij}\in\{2,3\}$},
\end{cases}
\end{align*}
for all $1\leq j\leq n$. Next, we are going to define a cohomology class 
$$[\alpha]\in H^1(C_2^k;(\R/{\Z})^n)$$
 where the $C_2^k$-module structure on $(\R/{\Z})^n$ is induced by $\rho$. For any $1\leq i\leq k$, we define $\alpha(g_i)=(s_1,...,s_n)$ where
\begin{align*}
s_j &=
\begin{cases}
0       & \text{if $(\widetilde{A})_{ij}\in\{0,2\}$}, 		\\
\frac{1}{2}      	& \text{if $(\widetilde{A})_{ij}\in\{1,3\}$},
\end{cases}
\end{align*}
for all $1\leq j\leq n$.  We define an $n$-dimensional crystallographic group of diagonal type $\Gamma_{\widetilde{A}}$ corresponding to the cohomology class $[\alpha]$. By (\ref{eq_cry_gen}), $\Gamma_{\widetilde{A}}$ is generated by $\{\iota(e_1),...,\iota(e_n),\gamma_1,...,\gamma_k\}$ where
$$\iota(e_j)=\begin{pmatrix}I_n & e_j\\ 0 & 1\end{pmatrix}$$
and $e_j$ is the $j^{th}$ column of the $n$-dimensional identity matrix for $1\leq j\leq n$ and
$$\gamma_i=\begin{pmatrix}\rho(g_i) & {\alpha}(g_i)\\ 0 & 1\end{pmatrix}$$
for $1\leq i\leq k$. By the construction, we can see that the characteristic matrix of $\Gamma_{\widetilde{A}}$ equals to $\widetilde{A}$. Notice that the holonomy group of $\Gamma_{\widetilde{A}}$ is not necessary isomorphic to $C_2^k$ because the representation $\rho:C_2^k\to \mathrm{GL}_n(\Z)$ is not necessary faithful. Conversely, again by construction, we can see that $\G_{A_{\Gamma}}\cong \G$.

If two characteristic matrices define isomorphic crystallographic groups, we will say that they are {\it equivalent}. In particular, the matrix obtained from swapping rows or columns of $A_\Gamma$ is equivalent to $A_\Gamma$.

\begin{examp}{\label{ex}}
	{\rm Let $\Gamma$ be the Bieberbach group enumerated in CARAT as ``min.19.1.1.7''. Let
		$$\gamma_1=\small \begin{pmatrix}
		-1 & 0 & 0 & 0 & 0 \\
		0 & -1 & 0 & 0 & 0 \\
		0 & 0 & 1 & 0 & \frac{1}{2}\\
		0 & 0 & 0 & -1 & \frac{1}{2}\\
		0 & 0 & 0 & 0 & 1
		\end{pmatrix}\,\,\,\,\text{and}\,\,\,\,\,\gamma_2=\begin{pmatrix}
		1 & 0 & 0 & 0 & \frac{1}{2} \\
		0 & 1 & 0 & 0 & 0 \\
		0 & 0 & -1 & 0 & 0\\
		0 & 0 & 0 & -1 & 0\\
		0 & 0 & 0 & 0 & 1\end{pmatrix}$$
		be non-lattice generators of $\Gamma$. The holonomy group of $\Gamma$ is $C_2^2$ which is generated by  $h_1=p(\gamma_1)$ and $h_2=p(\gamma_2)$ where $p:\Gamma\rightarrow C_2^2$ be the projection  defined in (\ref{eq_ses}). Using the same notations as in Remark \ref{remk_gen}, we have
		$$A=\small \begin{pmatrix}
		2 & 2 & 1 & 3\\
		1 & 0 & 2 & 2
		\end{pmatrix}$$
		We obtain the third row of $A_\Gamma$ by using Lemma \ref{lem_sum}.
		$$\begin{pmatrix} 2&2&1&3 \end{pmatrix}\star \begin{pmatrix} 1&0&2&2 \end{pmatrix}=\begin{pmatrix} 3&2&3&1 \end{pmatrix}$$
		Thus
		$$A_\Gamma=\small \begin{pmatrix}
		2 & 2 & 1 & 3\\
		1 & 0 & 2 & 2\\
		3 & 2 & 3 & 1
		\end{pmatrix}$$
		A simple calculation shows that
		$$\gamma_1\gamma_2=\small \begin{pmatrix}
		-1 & 0 & 0 & 0 & \frac{1}{2}-1 \\
		0 & -1 & 0 & 0 & 0 \\
		0 & 0 & -1 & 0 & \frac{1}{2}\\
		0 & 0 & 0 & 1 & \frac{1}{2}\\
		0 & 0 & 0 & 0 & 1
		\end{pmatrix}$$
		Therefore we verify that the third row of $A_\Gamma$ is indeed equal to (3 2 3 1).}
\end{examp}

\begin{examp}
{\normalfont Let $$\widetilde{A}  =\small \begin{pmatrix}
		2 & 1 & 2 & 3\\
		1 & 2 & 0 & 2\\
		3 & 3 & 2 & 1
		\end{pmatrix}$$
be a matrix generated by its first two rows. We are going to define its associated $4$-dimensional crystallographic group of diagonal type. We can assume $g_1$ and $g_2$  generate  $C_2^2$. First, we define a representation $\rho:C_2^2\to \mathrm{GL}_4(\Z)$ where
		$$\rho(g_1)=diag(-1,1,-1,-1)\,\,\,\,\,\,\,\text{and}\,\,\,\,\,\,\,\rho(g_2)=diag(1,-1,1,-1)$$
Next, we define the cohomology class $[\alpha] \in H^1(C_2^2,(\R/\Z)^4)$ by
$$\alpha(g_1)=\left(0,\tfrac{1}{2},0,\tfrac{1}{2}\right)\,\,\,\,\,\,\,\text{and}\,\,\,\,\,\,\,\alpha(g_2)=\left(\tfrac{1}{2},0,0,0\right)$$
Thus the characteristic matrix $\widetilde{A}$ defines a $4$-dimensional crystallographic group $ \G_{\widetilde{A}}$ where its non-lattice generators are
$$\gamma'_1=
\small\begin{pmatrix}
-1&0&0&0&0\\
0&1&0&0&\frac{1}{2}\\
0&0&-1&0&0\\
0&0&0&-1&\frac{1}{2}\\
0&0&0&0&1
\end{pmatrix}\,\,\,\,\,\,\,\text{and}\,\,\,\,\,\,\,
\gamma'_2=
\small\begin{pmatrix}
1&0&0&0&\frac{1}{2}\\
0&-1&0&0&0\\
0&0&1&0&0\\
0&0&0&-1&0\\
0&0&0&0&1
\end{pmatrix}
$$
Comparing the group $ \G_{\widetilde{A}}$ with $\Gamma$ defined in Example \ref{ex}, observe that we have
$$\gamma'_i=
\begin{pmatrix}
1&0&0&0&0\\
0&0&1&0&0\\
0&1&0&0&0\\
0&0&0&1&0\\
0&0&0&0&1
\end{pmatrix}\gamma_i
\begin{pmatrix}
1&0&0&0&0\\
0&0&1&0&0\\
0&1&0&0&0\\
0&0&0&1&0\\
0&0&0&0&1
\end{pmatrix}
$$
for $i=1,2$ and $\gamma_i$ are the elements defined in Example \ref{ex}. So $\Gamma\cong \G_{\widetilde{A}}$ and  hence $\widetilde{A}$ is equivalent to $A_\Gamma$. This result is not surprising because $\widetilde{A}$ can be obtained by swapping the second and the third columns of $A_\Gamma$.
}
\end{examp}

Next, we derive some properties of the matrix $A_\Gamma$.

\begin{lem}{\label{lem_entry=1}}
	 Let $\Gamma$ be a crystallographic group of diagonal type and let 
	$$[(\alpha_1, \dots, \alpha_n)]\in\bigoplus_{1\leq i\leq n} H^1(C_2^k;\R/M_i)$$
be the cohomology class corresponding to the standard extension of $\Gamma$ where $M_i\cong\Z$ for all $i=1,...,n$. Then $A_\Gamma(h,M_j)=1$ if and only if $0\neq res_{\langle h\rangle}[\alpha_j]\in H^1( C_2;\R/M_j)$.
\end{lem}
\begin{proof}
	Let $\Gamma'$ be the group corresponding to $res_{\langle h\rangle}[\alpha_j]\in H^1(C_2;\R/M_j)$. It then fits into the exact sequence
	$$0\to M_j\cong\Z\to\Gamma'\xrightarrow{p}\langle h\rangle\cong C_2.$$
	Observe that $res_{\langle h\rangle}[\alpha_j]\neq [0]$ if and only if $\Gamma'\cong\Z$. 

	Let $\rho:C_2^k\rightarrow \mathrm{GL}(\Z^n)$ be the holonomy representation of $\Gamma$ and recall that 
	$$({\alpha}_1, \dots, {\alpha}_n): C_2^k\rightarrow\left\{0,\tfrac{1}{2}\right\}^n$$
	 is the chosen vector system for $\Gamma$.Since $\Gamma$ is a crystallographic group of diagonal type, we can decompose $\rho$ as $\rho=\rho_1\oplus\cdots\oplus\rho_n$ such that $\rho(h)=diag(\rho_1(h),...,\rho_n(h))$. By (\ref{eq_cry_gen}), we have
	$$\Gamma'\cong\left\langle \begin{pmatrix} \rho_j(h) & {\alpha}_j(h)\\ 0 & 1\end{pmatrix}, \begin{pmatrix}1&1\\0&1\end{pmatrix}\right\rangle$$
	If $(\rho_j(h),\alpha_j(h))=(1,\frac{1}{2})$, then  clearly $\Gamma'\cong\Z$. 
	
	Suppose now that $\Gamma'\cong\Z$, but $(\rho_j(h),\alpha_j(h))\ne(1,\frac{1}{2})$. Then, $(\rho_j(h),\alpha_j(h))=(1,0)$. There is $q\in\{1,..,n\}$, such that $h$ acts nontrivially on $M_q$. Since $res_{\langle h\rangle}[\alpha_j]\neq [0]$, it follows that $res_{\langle h\rangle}[\alpha_j\oplus \alpha_q]\neq [0]$.  By (\ref{eq_cry_gen}), the crystallographic group $\Gamma''$ corresponding to $res_{\langle h\rangle}[\alpha_j\oplus \alpha_q]$, contains the element
	$$\begin{pmatrix} \rho_q(h) & 0& {\alpha}_q(h)\\0 & \rho_j(h)& {\alpha}_j(h)\\ 0& 0&1\end{pmatrix}=\begin{pmatrix} -1 & 0& {\alpha}_q(h)\\0 & 1& 0\\ 0& 0&1\end{pmatrix}$$
	which is of order 2. This implies that $res_{\langle h\rangle}[\alpha_j\oplus \alpha_q]= [0]$, which is a contradiction.

	Hence, we conclude that $\Gamma'\cong\Z$ if and only if $A_\Gamma(h,M_j)=1$.
\end{proof}

\begin{lem}{\label{lem_row_A}}Let $A$ be a $k\times n$-dimensional matrix with entries in $\{0, 1,2,3\}$ and  $\widetilde{A}$ be the $(2^k-1)\times n$-dimensional matrix generated by $A$.  Then
\begin{itemize}
\item[$(i)$]  Then $\G_{\widetilde{A}}$ has a torsion element if and only if $\widetilde{A}$  has a row where all entries are not equals to 1.
\item[$(ii)$] The holonomy group of $\G_{\widetilde{A}}$ is not isomorphic to $C_2^k$ if and only if $\widetilde{A}$ has a row where every  entry is either equal to 0 or 1. 
\end{itemize}
\end{lem}
\begin{proof}
By construction, $\widetilde{A}$  defines a cohomology class 
$$[\alpha]=[(\alpha_1, \dots, \alpha_n)]\in \bigoplus_{1\leq j\leq n} H^1(C_2^k;\R/M_j)$$
where $M_j\cong\Z$ for all $j=1,...,n$. By \cite[Theorem 3.1]{Szc cry}, $\Gamma$ has a torsion element if and only if there exists $g\in C_2^k$ such that 
$$res_{\langle g\rangle}[\alpha]=res_{\langle g\rangle}[\alpha_1]\oplus \cdots \oplus res_{\langle g\rangle}[\alpha_n]=0.$$
Hence, $\G_{\widetilde{A}}$ has a torsion element if and only if $res_{\langle g\rangle}[\alpha_j]=0$ for all $j=1,...,n$.  By Lemma \ref{lem_entry=1}, we can conclude that  $\G_{\widetilde{A}}$  has a torsion element if and only if $(\widetilde{A})_{ij} \neq 1$ for all $j=1,...,n$, where the $i$-th row corresponds to the element $g$.

Next, we note that the holonomy group of  $\G_{\widetilde{A}}$  is not isomorphic to $C_2^k$ if and only if there exists $g\in C_2^k$ such that $g$ acts trivially on $\R^n/\Z^n$. By construction, $g$ acts trivially on $\R^n/\Z^n$ if and only if $(\widetilde{A})_{ij}\in\{0,1\}$ for all $j=1,...,n$., where the $i$-th row corresponds to the element $g$.
\end{proof}

\begin{prop}{\label{prop_mat_prop}} Let $\Gamma$ be an $n$-dimensional Bieberbach group of diagonal type with holonomy group $C_2^k$. Then every row of $A_{\G}$ has an entry equal to 1 and an entry equal to either 2 or 3.
\end{prop}
\begin{proof}
The proof follows immediately from Lemma \ref{lem_row_A} and the discussion after Remark \ref{remk_gen}.
\end{proof}

Let $\omega\in H^2(C_2^k;\Z)$. It will be convenient for computations to define 
$$\mathcal R(\omega)=\{g\in C^k_2 \;|\; \,res_{\langle g\rangle}(\omega)\neq 0\}.$$

\begin{remk}{\label{remk_R(a)}}
{\normalfont Let $\Gamma$ be an $n$-dimensional crystallographic group with holonomy $C_2^k$. Let $\omega\in H^2(C_2^k;\Z^n)$ be the cohomology class corresponding to standard extension of $\Gamma$. Note that  $\omega=(\omega_1, \dots, \omega_n)$ where $\omega_j\in H^2(C_2^k; M_j)$ and $M_j\cong \Z$ for all $j=1,...,n$. By (\ref{eq_21}) and Lemma \ref{lem_entry=1}, for any $j\in\{1,...,n\}$,  we have $A_\Gamma(g,M_j)=1$ if and only if $g\in\mathcal R(\omega_j)$.}
\end{remk}

\begin{prop}{\label{prop_R(a)_tri}}  Let $ 0 \neq\omega\in H^2(C^k_2;\Z)$ where $C^k_2$ acts trivially on $\Z$. Then we have $|\mathcal R(\omega)|=2^{k-1}$.
\end{prop}
\begin{proof}
Under the isomorphism discussed in Section \ref{sec_bieb},  $\omega$ corresponds to $[\alpha]\in H^1(C_2^k;\R/\Z)$. By Lemma \ref{lem_coho}, we can assume $\alpha(g)\in\{0,\frac{1}{2}\}$ for all $g\in C_2^k$. Thus
$$|\mathcal R(\omega)|=\left|\left\{g\in C_2^k\,\big|\,\alpha(g)=\tfrac{1}{2}\right\}\right|$$
Since $[\alpha]\neq 0$, there exists $g\in C_2^k$ such that $\alpha(g)=\frac{1}{2}$. Let $C_2^k=H\sqcup gH$ where $H\leq C_2^k$ and $H\cong C_2^{k-1}$. For any $h\in H$, we have
\begin{align*}
		\alpha(gh) =\alpha(g)+\alpha(h)&=
		\begin{cases}
		0		& \text{if $\alpha(h)=\tfrac{1}{2}$},\\
		\tfrac{1}{2}       & \text{if $\alpha(h)=0$}.  	
		\end{cases}
\end{align*}
Thus $|\mathcal R(\omega)|=|H|=|C_2^{k-1}|=2^{k-1}$.
\end{proof}

\begin{prop}{\label{prop_R(a)_nontri}} Let $0\neq \beta\in H^2(C^k_2;\Z)$ where $C^k_2$ acts non-trivially on $\Z$ via $\rho: C_2^k\rightarrow \mathrm{GL}(\Z)$. Then $|\mathcal R(\beta)|=2^{k-2}$.
\end{prop}
\begin{proof}
	Define $\omega=res_{ker(\rho)}(\beta)\in H^2(ker(\rho);\Z)$. Since $\beta\neq 0$ and $H^2(\langle g\rangle ; \Z)=0$ for all $g\not\in ker(\rho)$, it follows that $\omega\neq 0$ and 	
	$$|\mathcal R(\beta)|=|\{g\in C^k_2 \,|\, \,res_{\langle g\rangle}(\beta)\neq 0\}|=|\{h\in ker(\rho)\cong C^{k-1}_2\,|\, res_{\langle h\rangle} (\omega)\neq 0\}|=|\mathcal R(\omega)|$$
	By Proposition \ref{prop_R(a)_tri}, we have $|\mathcal R(\beta)|=2^{k-2}$.
\end{proof}

\begin{remk}{\label{remk_col=1}}
	{\normalfont Let $\Gamma$ be an $n$-dimensional Bieberbach group of diagonal type with holonomy $C_2^k$. Let 
	$$\omega_1 \oplus \cdots  \oplus \omega_n\in\bigoplus_{1\leq j\leq n} H^2(C_2^k; M_j)$$
	be the cohomology class corresponding to standard extension of $\Gamma$, where $M_j\cong\Z$ for all $j=1,...,n$. By Proposition \ref{prop_R(a)_tri} and Proposition \ref{prop_R(a)_nontri}, for any $j\in\{1,...,n\}$, we have $|\mathcal R(\omega_j)|=2^{k-2}$ or $2^{k-1}$. By Remark \ref{remk_R(a)}, we can conclude that in every column of $A_\Gamma$, there exist at least $2^{k-2}$ entries equal to $1$.}
\end{remk}

\begin{prop}{\label{prop_col=ker}} Let $\omega\in H^2(C_2^k;\Z)$ where $C_2^k$ acts non-trivially on $\Z$ via $\rho: C_2^k\rightarrow \mathrm{GL}(\Z)$. If $\mathcal T\subseteq \mathcal R(\omega)$ with $|\mathcal T|\geq 2^{k-3}+1$, then $\langle \mathcal T\rangle =\langle\mathcal R(\omega)\rangle = ker(\rho)$. 
\end{prop}
\begin{proof}
	Since $\mathcal T\subseteq \mathcal R(\omega)\subseteq ker(\rho)$, we have $\langle \mathcal T\rangle\leq ker(\rho)$. We assume by contradiction that $\langle \mathcal T\rangle\lneqq ker(\rho)$. Since $|\mathcal T|\geq 2^{k-3}+1$, we have $\langle \mathcal T\rangle\cong C_2^{k-2}$. Consider $\theta=res_{\langle \mathcal T\rangle}(\omega)\in H^2(\langle \mathcal T\rangle;\Z)$. Recall that $\mathcal R(\theta)=\{h\in \langle \mathcal T\rangle \,|\,res_{\langle h \rangle}(\theta)\neq 0\}$. By Proposition \ref{prop_R(a)_tri}, we have $|\mathcal R(\theta)|=2^{k-3}$. Since $\mathcal T\subseteq \mathcal R(\theta)$, we have
	$$2^{k-3}+1\leq |\mathcal T|\leq|\mathcal R(\theta)|=2^{k-3}$$
	which is a contradiction.
\end{proof}

\begin{cor}{\label{cor_col=col}}
	Let $\Gamma$ be an $n$-dimensional Bieberbach group of diagonal type with holonomy $C_2^k$. Let 
	$$\omega_1 \oplus \cdots  \oplus \omega_n \in H^2(C_2^k; M_1)\oplus\cdots\oplus H^2(C_2^k; M_n)$$
	 be the cohomology class corresponding to standard extension of $\Gamma$ where $M_z\cong\Z$ for $z=1,...,n$. Let $\rho_z:C_2^k\rightarrow \mathrm{GL}(M_z)$ be the representations given by the $C_2^k$-action on $M_z$ for all $z=1,...,n$. If there exists $i,j\in\{1,...,n\}$ such that $\rho_i$ and $\rho_j$ are non-trivial representations and there exists a subset $\mathcal T\subseteq \mathcal R(\omega_i)\cap \mathcal R(\omega_j)$ such that $|\mathcal T|\geq 2^{k-3}+1$, then $\mathcal R(\omega_i)=\mathcal R(\omega_j).$
\end{cor}
\begin{proof} Let $i,j\in\{1,...,n\}$ such that $\rho_i$ and $\rho_j$ are non-trivial representation and there exists a subset $\mathcal T\subseteq \mathcal R(\omega_i)\cap \mathcal R(\omega_j)$ such that $|\mathcal T|\geq 2^{k-3}+1$. By Proposition \ref{prop_col=ker}, we have $ker(\rho_i)=\langle \mathcal T \rangle = ker(\rho_j)$. Since $\mathcal R(\omega_i)\subseteq ker(\rho_i)=\langle \mathcal \mathcal T\rangle$ and $\mathcal R(\omega_j)\subseteq ker(\rho_j)=\langle \mathcal T\rangle$, every element belongs to $\mathcal R(\omega_i)\cup \mathcal R(\omega_j)$ can be expressed as a combination of elements of $\mathcal T$. Let $x=t_1\cdots t_s\in \mathcal R(\omega_i)$ where $t_1,...,t_s\in \mathcal T$. We have $A_\Gamma(x,M_i)=\bigstar_{z=1}^{s}A_\Gamma(t_z,M_i)$ and $A_\Gamma(x,M_j)=\bigstar_{z=1}^{s}A_\Gamma(t_z,M_j)$. Since $\mathcal T\subseteq \mathcal R(\omega_i)\cap \mathcal R(\omega_j)$, we have $A_\Gamma(t,M_i)=A_\Gamma(t,M_j)=1$ for each $t\in \mathcal T$. Hence we have
	$$A_\Gamma(x,M_i)=\bigstar_{z=1}^{s}A_\Gamma(t_z,M_i)=\bigstar_{z=1}^{s}A_\Gamma(t_z,M_j)=A_\Gamma(x,M_j)$$
	By Remark \ref{remk_R(a)}, we have $x\in \mathcal R(\omega_i)$ if and only if $x\in \mathcal R(\omega_j)$. Hence $\mathcal R(\omega_i)=\mathcal R(\omega_j)$.
\end{proof}

\section{Vasquez invariant  of diagonal type}{\label{sec_vas}}
 
 In this section, we give an alternative definition of the Vasquez invariant for Bieberbach groups of diagonal type which is more suited for our purposes.

\begin{defn}
{\normalfont A {\it $G$-lattice} is any $\Z G$-module whose underlying abelian group is isomorphic to a free abelian group $\Z^n$ for some $n\geq 1$. Let $M$ be a $G$-lattice where $\{e_1,...,e_n\}$ is a generating set of $M$. We say $M$ is a {\it diagonal $G$-lattice} if $g\cdot e_i=\pm e_i$ for all $g\in G$ and $i\in\{1,...,n\}$. We say $M$ is a {\it faithful $G$-lattice} if $G$ acts faithfully on $M$.}
\end{defn}

By \cite[Theorem 3]{Szc}, there is a different way to define  the Vasquez invariant of finite groups.

\begin{defn}{\label{def-vas}}
	{\normalfont Let $G$ be a finite group and let $L$ be a $G$-lattice. An element $\omega \in H^2(G;L)$ is said to be {\it special} if its extension defines a Bieberbach group. The $G$-lattice $L$ is said to have {\it property $\mathbb{S}$} if for any $G$-lattice $M$ and any special element $\omega\in H^2(G;M)$, there exists a $G$-homomorphism $f:M\rightarrow L$ such that $f^*:H^2(G; M)\rightarrow H^2(G; L)$ maps $\omega$ to another special element $f^*(\omega)\in H^2(G;L)$. 
	The Vasquez invariant of a finite group $G$ is then
	$$n(G)=\min\{ rank_\Z(L) \;|\; \text{$L$ is $G$-lattice with property $\mathbb{S}$}\}.$$}
\end{defn}

\begin{lem}{\label{lem-diag}} Let $\Gamma$ be an $n$-dimensional Bieberbach group of diagonal type and let $\omega \in H^2(G;\Z^n)$ be the cohomology class corresponding to the standard extension of $\Gamma$. Let $f:\Z^n\rightarrow M$ be a $G$-epimorphism such that $f^*(\omega)$ is special. Then $f^*(\omega)$ defines a Bieberbach group of diagonal type.
\end{lem}
\begin{proof}
	Since $\Gamma$ is a Bieberbach group of diagonal type, we can assume $\Z^n$ is a diagonal $G$-lattice where $\{e_1,...,e_n\}$ is a basis of $\Z^n$ such that $g\cdot e_i=\pm e_i$ for all $g \in G$ and $i=1,...,n$. Let $M_{\R}=M\otimes \R$. Then $\{e_1,...,e_n\}$ form a basis for $\R^n$. Denote by $f_{\R}:\R^n\rightarrow M_{\R}$ the epimorphism induced by $f$.  Since the set $\{f_{\R}(e_1),...,f_{\R}(e_n)\}$ spans $M_{\R}$, there exists a subset $\{f_{\R}(e_{i_1}),...,f_{\R}(e_{i_t})\}$ that forms a basis for $M_{\R}$. The holonomy representation of the Bieberbach group defined by $f^*(\omega)\in H^2(G; M)$ is given by the $G$-action on $M$. Thus, it suffices to check that $G$ acts diagonally on $\{f_{\R}(e_{i_1}),...,f_{\R}(e_{i_t})\}$, for in this case the holonomy representation of $G$ on $M$ is conjugate in $\mathrm{GL}_t(\R)$ to a diagonal one.  Since $f_{\R}$ is a module homomorphism, for all $g\in G$ and for all $i\in\{1,...,n\}$ we have
$$g\cdot f_{\R}(e_i)=f_{\R}(g\cdot e_i)=f_{\R}(\pm e_i)=\pm f_{\R}(e_i)$$
Hence, $f_*(\omega)$ defines a Bieberbach group of diagonal type.
\end{proof}

As an immediate application we obtain the following.

\begin{prop}\label{quot_prop} Let $\Gamma$ be a Bieberbach group of diagonal type and let $N$ be normal subgroup of its lattice subgroup. If $\G/N$ is a Bieberbach group, then it is of diagonal type.
\end{prop}

\begin{thm}{\label{thm-diag}}  Let $G$ be an elementary abelian $2$-group. If $\Gamma$ is a Bieberbach group of diagonal type where its holonomy group is isomorphic to $G$, then the lattice subgroup  $L\subseteq\Gamma$ contains a normal subgroup $N$ such that $\Gamma/N$ is a Bieberbach group of diagonal type of dimension at most $n_d(G)$.
\end{thm}	
\begin{proof} The result follows from Proposition \ref{quot_prop} and Theorem \ref{def vas alg}.
\end{proof}

\begin{remk}{\label{remk_cd}}
{\normalfont  Combining Proposition \ref{quot_prop} and Theorem \ref{def vas alg},  we have  that $n_d(G)\leq n(G)$.} 
\end{remk}

\begin{remk}{\label{remk_coho_nd}}
{\normalfont 
Let $\Gamma$ be a Bieberbach group of diagonal type. Let $\omega\in H^2(G; L)$ be the cohomology class that defines $\Gamma$ where $G$ is an elementary abelian $2$-group and $L$ is a diagonal faithful $G$-lattice. By Theorem \ref{thm-diag}, there exists a normal subgroup $N\unlhd L$ such that $\Gamma/N$ is a Bieberbach group of diagonal type with dimensional at most $n_d(G)$. In other words, we can define a $G$-homomorphism $f:L\rightarrow L/N$ such that $f^*:H^2(G;L)\rightarrow H^2(G;L/N)$ maps $\omega$ to another special element $f^*(\omega)$ defining a Bieberbach group of diagonal type of dimension at most $n_d(G)$.}
\end{remk}

\begin{defn}{\label{def_sd}}
	{\normalfont Let $G$ be an elementary abelian $2$-group  and let $L$ be a diagonal $G$-lattice. An element $\omega \in H^2(G;L)$ is said to be a {\it diagonal special} element if its extension defines a Bieberbach group of diagonal type. We say a diagonal $G$-lattice $L$ has {\it property $\mathbb{S}_d$} if for any diagonal $G$-lattice $M$ and any diagonal special element $\omega \in H^2(G;M)$, there exists a $G$-homomorphism $f:M\rightarrow L$ such that $f^*:H^2(G;M)\rightarrow H^2(G;L)$ maps $\omega$ to another diagonal special element $f^*(\omega)$.}
\end{defn}

\begin{thm}
	Let $G$ be an elementary abelian $2$-group. Define $$n'_d(G)=\min\{rank_\Z(L) \;|\; \text{$L$ is a diagonal $\Z G$-lattice with property $\mathbb{S}_d$}\}$$
	Then we have $n_d(G)=n'_d(G)$.
\end{thm}
\begin{proof}
	By definition, it is clear that $n_d(G)\leq n'_d(G)$. Now we want to prove that $n'_d(G)\leq n_d(G)$. Let $L$ be a diagonal $G$-lattice of minimal rank with property $\mathbb{S}_d$. In other words, $rank_\Z(L)=n'_d(G)$. Let $M$ be any diagonal $G$-lattice and $\omega \in H^2(G;M)$ be any diagonal special element. Since $L$ has property $\mathbb{S}_d$ and by Definition \ref{def_sd}, there exists a $G$-homomorphism $g:M\rightarrow L$ such that $g^*(\omega)\in H^2(G;L)$ is a diagonal special element.
		
	First, we assume $L$ is a faithful diagonal $G$-lattice. Since $L$ is faithful, by Remark \ref{remk_coho_nd}, there exists a diagonal $G$-lattice $K$ with $rank_\Z(K)\leq n_d(G)$ and a $G$-homomorphism $h:L\rightarrow K$ such that $h^*(g^*(\omega))$ is a diagonal special element. Hence $K$ is a diagonal $G$-lattice with property $\mathbb{S}_d$. It follows that $n'_d(G)\leq rank_\Z(K)$. Therefore we have $n'_d(G)\leq rank_\Z(K)\leq n_d(G)$.
	
	Now assume $L$ is not a faithful diagonal $G$-lattice. Let $P$ be a faithful diagonal $G$-lattice. Consider the faithful diagonal $G$-lattice $L\oplus P$. We have $(g^*(\omega), 0)\in H^2(G;L)\oplus H^2(G;P)$ is a diagonal special element. By Remark \ref{remk_coho_nd}, there exists a diagonal $G$-lattice $N$ with $rank_\Z(N)\leq n_d(G)$ and a $G$-homomorphism $f:L\oplus P\rightarrow N$ such that $f^*((g^*(\omega),0))\in H^2(G; N)$ is a special element. Since
	$$Hom_G(L\oplus P,N)\cong Hom_G(L,N)\oplus Hom_G(P,N),$$
	 we can let $f=f_1\oplus f_2$ where $f_1\in Hom_G(L,N)$ and $f_2\in Hom_G(P,N)$. Therefore, $f^*((g^*(\omega),0))=(f_1)^*(g^*(\omega))$. Thus $N$ is a diagonal $G$-lattice with property $\mathbb{S}_d$. It follows that $n'_d(G)\leq rank_\Z(N)$. Hence, we have $n'_d(G)\leq n_d(G)$. 
\end{proof}

\section{Proofs of Theorems A and B}{\label{sec_thm_a_b}}
In this section, given a Bieberbach group $\Gamma$ of diagonal type, we analyse the characteristic matrix $A_{\Gamma}$ to determine whether there exists a normal subgroup such that the quotient is still a Bieberbach group. 

\begin{defn}
	{\normalfont Let $\Gamma$ be an $n$-dimensional Bieberbach group of diagonal type. The characteristic matrix $A_\Gamma$ is said to be {\it col-reducible} (by $i^{th}$ column) if after removing a column ($i^{th}$ column) from $A_\Gamma$, there still exists an entry equal to 1 in every row. We say $A_\Gamma$ is {\it col-irreducible} if it is not col-reducible.}
\end{defn}

\begin{lem}{\label{lem ker of homo}}
Let $f:N\rightarrow M$ be a $C_2^k$-homomorphism where $N=M_1\oplus\cdots\oplus M_n$ and $M_1,\dots , M_n$ are all $C_2^k$-lattices of rank one. Let $e_i$ be a generator of $M_i$ and $\rho_i:C_2^k\rightarrow \mathrm{GL}(M_i)$ be the representation defining  the $C_2^k$-action on $M_i$ for all $1\leq i \leq n$.  For $i\in\{1,...,n\}$, if there exists $b \ne 0$, $t\geq 2$, and a linearly independent set $\{f({e_{i_1}}), \dots  , f({e_{i_t}})\}$ such that 
$$b f({e_i})=a_{i_1}f({e_{i_1}})+\cdots + a_{i_t}f({e_{i_t}})$$
  where $a_{i_1},...,a_{i_t}\neq 0$, then $ker(\rho_{i_1})=\cdots=ker(\rho_{i_t})$.
\end{lem}
\begin{proof}
Assuming the hypothesis, for any $g\in C_2^k$ that acts trivially on $M_i$, we have
	$$a_{i_1}f({e_{i_1}})+\cdots +a_{i_t}f({e_{i_t}})=bf({e_i})=bf({g\cdot e_i})=g\cdot bf({e_i})=\sum_{z=1}^t a_{i_z}g\cdot f({e_{i_z}})=\sum_{z=1}^t a_{i_z}f({g\cdot e_{i_z}}).$$
Since $g\cdot f({e_{i_z}})=\pm f({e_{i_z}})$, this shows that $f({g\cdot e_{i_z}})=f({e_{i_z}})$ for all $z\in\{1,...,t\}$. It follows that $g\in ker(\rho_{i_z})$ for all $z\in\{1,...,t\}$. For each $h\in C_2^k$ that acts non-trivially on $M_i$, by similar calculation, we get 
	$$-a_{i_1}f({e_{i_1}})-\cdots -a_{i_t}f({e_{i_t}})=-bf({e_i})=f({h\cdot b e_i})=h\cdot bf({e_i})=\sum_{z=1}^t a_{i_z}h\cdot f({e_{i_z}})=\sum_{z=1}^t a_{i_z}f({h\cdot e_{i_z}})$$
	It follows that $f({h\cdot e_{i_z}})=f({-e_{i_z}})$ for all $z\in\{1,...,t\}$. Therefore $h\not\in ker(\rho_{i_z})$ for all $z\in\{1,...t\}$. Hence we can conclude that $ker(\rho_{i_1})=\cdots=ker(\rho_{i_t})$.
\end{proof}

\begin{cor}{\label{r col-irred and ker of rep}} Let $\Gamma$ be an $n$-dimensional Bieberbach group of diagonal type with holonomy $C_2^k$ and let 
$$\omega=\omega_1 \oplus \cdots  \oplus \omega_n\in H^2(C_2^k;M_i)\oplus \cdots\oplus H^2(C_2^k;M_n)$$ 
be the corresponding cohomology class where $M_i\cong \Z$ for all $i=1,...,n$. Let $\rho_i:C_2^k\rightarrow \mathrm{GL}(M_i)$ be the representation given by the $C_2^k$-action on $M_i$ for all $1\leq i \leq n$. If $ker(\rho_i)\neq ker(\rho_j)$ for all $i\neq j$ and $A_\Gamma$ is col-irreducible, then there does not exist a $C_2^k$-homomorphism $f:\Z^n\rightarrow\Z^s$ where $s<n$ such that $f^*(\omega)\in H^2(G; \Z^s)$ is special.
\end{cor}
\begin{proof}
	Assume by contradiction that there exists a $C_2^k$-homomorphism $f:\Z^n\rightarrow\Z^s$ where $s<n$ such that $f^*(\omega)=f^*(\omega_1)\oplus\cdots\oplus f^*(\omega_n)$ is special. Let $e_i$ be the generator of $M_i$ for each $1\leq i\leq n$.  Since $\{f(e_{1}), \dots , f(e_{n})\}$ spans $f(\Z^n)$,  there is a linear independent subset  $\{f(e_{s_1}), \dots ,f(e_{s_t})\}$ that spans a finite index sublattice in $f(\Z^n)$.
	
	Suppose $f({e_i})=0$ for some $i$. Then $f$ factors through the projection  $\pi_i:\Z^n\to \Z^{n-1}$ mapping $e_i$ to $0$. Denote the resulting homomorphism  $f_i:\Z^{n-1}\to \Z^s$. Let $\Gamma_i$ be the group defined by $\pi_i^*(\omega)\in H^2(G; \Z^{n-1})$.  Let $\Gamma'$ be the Bieberbach group defined by  $f^*(\omega)$. Note that  the homomorphism from $\Gamma$ to $\Gamma'$ defined by $f$ factors through the homomorphism $\psi_i:\Gamma_i\to \Gamma'$ defined by $f_i$. Since $\Gamma'$ is torsion-free and  $\ker (\psi_i)=\ker (f_i)$, it follows that $\Gamma_i$  is torsion-free and hence a Bieberbach group (of diagonal type). By Proposition \ref{prop_mat_prop}, there exists an entry equal to 1 in every row of $A_{\Gamma_i}$. But $A_{\Gamma_i}$ is the matrix obtained from $A_\Gamma$ by deleting the $i$-th column. This is a contradiction to the fact that $A_\Gamma$ is col-irreducible.

	 It follows that for each $i$, $f({e_i})\ne0$. So, there exists $b_i\ne 0$ and nonzero $a_{i_1},...,a_{i_t}\in\Z$ such that $b_if({e_i})=a_{i_1}f({e_{s_1}})+\cdots+a_{i_t}f({e_{s_t}})$. By Lemma \ref{lem ker of homo}, if $t\geq 2$, we have $ker(\rho_{s_1})=\cdots=ker(\rho_{s_t})$, which contradicts the fact that the kernels of $\rho_i$ are all distinct for $1\leq i\leq n$. It follows that for each $1\leq i\leq n$, there exists $b_i\ne 0$, such  that  $b_if({e_i})=a_{i'}f({e_{i'}})\ne 0$  for some  $i'\in\{s_1,...,s_t\}$ and $a_{i'}\ne 0$.   Since $t\leq s<n$, there exists  $1\leq i\leq n$, so that $i\ne i'$. For any $g\in G$, we have
$$ \pm b_if(e_i )=b_if(g\cdot e_i )= a_{i'}f(g\cdot e_{i'})=\pm a_{i'}f(e_{i'}).$$
This shows that $\ker (\rho_i)=\ker (\rho_{i'})$, which is a contradiction.
\end{proof}

\begin{defn}
Let $M$ be an $m\times n$-matrix with entries in $\{0,1,2,3\}$. Define the map $\phi:\{0,1,2,3\}\to \{p,q\}$ such that $\phi(0)=\phi(1)=p$ and $\phi(2)=\phi(3)=q$. Define $\Phi(M)$ to be the $m\times n$-matrix such that $[\Phi(M)]_{ij}=\phi(M_{ij})$. 
\end{defn}
\begin{lem}{\label{lem_ker}}
Let $A$ be a $k\times n$-dimensional matrix with entries in $\{0,1,2,3\}$. Let $\G_{\widetilde{A}}$ be the crystallographic group of diagonal type obtained by $A$. Let
$$\omega_1 \oplus \cdots  \oplus \omega_n\in H^2(C_2^k; M_1)\oplus\cdots\oplus H^2(C_2^k; M_n)$$
be the cohomology class corresponding to standard extension of $\G_{\widetilde{A}}$. Let $C_2^k$ acts on $M_z$ via $\rho_z:C_2^k\to GL(\Z)$ for $1\leq z\leq n$. Then the $i^{th}$ column of $\Phi(A)$ is equal to the $j^{th}$ column of $\Phi(A)$ if and only if $ker(\rho_i)=ker(\rho_j)$.
\end{lem}
\begin{proof}
Let $g_1,...,g_k$ be the generators of $C_2^k$. Assume without loss of generality, we have $A_{ij}=A_\Gamma(g_i,M_j)$. If $ker(\rho_i)\neq ker(\rho_j)$, then there exists $r\in\{1,...,k\}$ such that $g_r$ acts differently on $M_i$ from $M_j$. We may assume $g_r$ acts trivially on $M_i$ and acts non-trivially on $M_j$. By equation (\ref{lab_action}), we have $A_\Gamma(g_r,M_i)\in\{0,1\}$ and $A_\Gamma(g_r,M_j)\in\{2,3\}$. Thus by definition, we get $\Phi(A)_{ri}=\phi(A_\Gamma(g_r,M_i))=p$ and $\Phi(A)_{rj}=\phi(A_\Gamma(g_r,M_j))=q$. Therefore the $i^{th}$ column of $\Phi(A)$ is not equal to the $j^{th}$ column of $\Phi(A)$. 

Now, we assume the the $s^{th}$ column of $\Phi(A)$ is not equal to the $i^{th}$ column of $\Phi(A)$. Without loss of generality, we may assume there exists $r\in\{1,...,k\}$ such that $\Phi(A)_{rs}=\Phi(A_\Gamma(g_r,M_s))=p$ and $\Phi(A)_{ri}=\Phi(A_\Gamma(g_r,M_i))=q$. Hence $A_\Gamma(g_r,M_i)\in\{0,1\}$ and $A_\Gamma(g_r,M_s)\in\{2,3\}$. By equation (\ref{lab_action}), the holonomy group element $g_r$ acts trivially on $M_s$ and non-trivially on $M_i$. Thus we have $ker(\rho_s)\neq ker(\rho_i)$.
\end{proof}

\begin{prop}{\label{prop_col_irred}}
Let $\Gamma$ be an $n$-dimensional Bieberbach group of diagonal type with holonomy $C_2^k$ and let $A_\Gamma$ be its characteristic matrix. The matrix $A_\Gamma$ is col-irreducible if and only if $A_\Gamma$ can be transformed to a matrix $Y=\begin{pmatrix}
	X\\N
	\end{pmatrix}$ by swapping rows and columns of $A_\Gamma$, where $X$ is an $n\times n$-matrix with all diagonal entries equal to 1 and all other entries not equal to 1.
\end{prop}
\begin{proof}
First, we assume $A_\Gamma$ can be transformed to $Y$ and we want to prove that $A_\Gamma$ is col-irreducible. It is sufficient to show that $Y$ is col-irreducible. For any $i\in\{1,...,n\}$, if we remove the $i^{th}$ column of  $Y$, then the $i^{th}$ row of the of new matrix will not have an entry equal to 1. Hence, we conclude that  $Y$ is col-irreducible. 
	
	Next, we assume $A_\Gamma$ is col-irreducible. For any $i\in\{1,...,n\}$, we consider the $i^{th}$ column of $A_\Gamma$. By definition of col-irreducibility, if we remove the $i^{th}$ column of $A_\Gamma$, there is $r_i\in\{1,...,2^k-1\}$ such that the $r_i^{th}$ row of the new matrix does not have entries equal to 1. By Proposition \ref{prop_mat_prop}, the $r_i^{th}$ row of $A_\Gamma$ has at least one entry equal to 1. Therefore we can conclude that $(A_\Gamma)_{r_i,i}=1$ and $(A_\Gamma)_{r_i,s}\neq 1$ for all $s\neq i$. Notice that we have $r_i\neq r_j$ for all $i\neq j$. We obtain a new matrix $A'_\Gamma$ by swapping the rows of $A_\Gamma$ such that the $i^{th}$ row of $A'_\Gamma$ equals to the $r_i^{th}$ row of $A_\Gamma$. Hence $A'_\Gamma$ can be transformed to $Y$.
\end{proof}

\begin{lem}{\label{lem_colred}}
Let $\Gamma$ be an $n$-dimensional Bieberbach group of diagonal type with holonomy $C_2^k$ and let $\Gamma\cap\R^n=\langle e_1,...,e_n\rangle\cong\Z^n$. If $A_\Gamma$ is col-reducible by $i^{th}$ column, then $\Gamma/\langle e_i\rangle$ is a Bieberbach group of diagonal type.
\end{lem}
\begin{proof}
We define a $C_2^k$-homomorphism $f:\Z^n\to\Z^{n-1}$ such that $f(e_i)=0$ and $f(e_j)=e_j$ for all $j\neq i$. Let 
	$$\omega=\omega_1\oplus\cdots\oplus\omega_n\in H^2(C_2^k; M_1)\oplus\cdots\oplus H^2(C_2^k; M_n)$$
	be the cohomology class corresponding to standard extension of $\Gamma$ where $M_i\cong\Z$ for all $i=1,...,n$. We have $$f^*(\omega)=\omega_1\oplus\cdots\oplus\omega_{i-1}\oplus \omega_{i+1} \oplus\cdots\oplus \omega_n$$
	 and $f^*(\omega)$ defines the Bieberbach group $\Gamma/\langle e_i\rangle$. The characteristic matrix corresponding to $\Gamma/\langle e_i\rangle$ can be obtained by removing the $i^{th}$ column of $A_\Gamma$. Since $A_\Gamma$ is col-reducible by $i^{th}$ column, every row of $A_{\Gamma/\langle e_i\rangle}$ has at least one entry equal to 1. By Lemma \ref{lem_row_A}, $\Gamma/\langle e_i\rangle$ is a Bieberbach group. By Lemma \ref{lem-diag}, $\Gamma/\langle e_i\rangle $ is a Bieberbach group of diagonal type.
\end{proof}

In the next two propositions, we obtain upper and lower bounds for the diagonal Vasquez number.

\begin{prop}{\label{p col-red}} Let $\Gamma$ be an $n$-dimensional Bieberbach group of diagonal type with holonomy $C_2^k$. If $n> 5\cdot 2^{k-3}+1$ and $k\geq 2$, then $A_\Gamma$ is col-reducible.
\end{prop}
\begin{proof} 
	First, we consider the case where $k=2$. Let $\Gamma$ be an $n$-dimensional Bieberbach group of diagonal type with holonomy $C_2^2$ and $n\geq 4$. Consider its characteristic matrix $A_\Gamma$. The matrix $A_\Gamma$ has 3 rows and at least 4 columns. Thus it cannot be col-irreducible by Proposition \ref{prop_col_irred}. It follows that $A_\Gamma$ is col-reducible.
	
	Now, assume $k\geq 3$ and let $\Gamma$ be an $n$-dimensional Bieberbach group of diagonal type with holonomy $C_2^k$ and $n\geq 5\cdot 2^{k-3}+2$. Let 
	$$\omega_1\oplus\cdots\oplus\omega_n\in H^2(C_2^k; M_1)\oplus\cdots\oplus H^2(C_2^k; M_n)$$
	 be the cohomology class corresponding to the standard extension of $\Gamma$, where $M_i\cong\Z$ for all $i=1,...,n$. Assume by contradiction that $A_\Gamma$ is col-irreducible. By Proposition \ref{prop_col_irred}, we assume $A_\Gamma=\begin{pmatrix}
	X\\N
	\end{pmatrix}$ where $X$ is an $n\times n$-matrix with all diagonal entries equal to 1 and all the others not equal to 1 and $N$ is a matrix with $2^k-1-n$ rows. Since $C_2^k$ acts faithfully on $\Z^n$, there exists $i,j\in\{1,...,n\}$ such that $C_2^k$ acts non-trivially on both $M_i$ and $M_j$ where $i\neq j$. Consider the $i^{th}$ and $j^{th}$ columns of $N$. By Proposition \ref{prop_R(a)_nontri}, the $i^{th}$ and $j^{th}$ columns of $N$  each have exactly $2^{k-2}-1$ entries equal to 1. Since $k\geq 3$, this ensures that the $i^{th}$ and $j^{th}$ columns of $N$ each have  at least one entry equal to 1. Define
	$$z=|\{m\in\{1,...,2^k-1-n\}\,|\,N_{m,i}=N_{m,j}=1\}|$$
	and observe that we have
	$$2^{k-2}-1-z=|\{m\in\{1,...,2^k-1-n\}\,|\,N_{m,i}=1,\,N_{m,j}\neq 1\}|$$
	and
	$$2^{k-2}-1-z=|\{m\in\{1,...,2^k-1-n\}\,|\,N_{m,i}\neq1,\,N_{m,j}= 1\}|$$
	Since $N$ has $2^k-1-n$ rows, we have
	$$2(2^{k-2}-1-z)+z\leq 2^k-1-n$$
	By rearranging the above inequality, we get
	$$n-2^{k-1}-1\leq z$$
	Since $n\geq 5\cdot 2^{k-3}+2$, it follows that
	$$z\geq5\cdot 2^{k-3}+2-2^{k-1}-1=2^{k-3}+1$$
	Thus
	$$	|\{g\in C_2^k|A_\Gamma(g,M_i)=A_\Gamma(g,M_j)=1\}| \geq 2^{k-3}+1	$$
	By Corollary \ref{cor_col=col}, we have $\mathcal R(\omega_i)=\mathcal R(\omega_j)$
	and hence by Remark \ref{remk_R(a)}
	$$\{g\in C_2^k | A_\Gamma(g,M_i)=1\}=\{g\in C_2^k | A_\Gamma(g,M_j)=1\}.$$
	This is a contradiction to the fact that  all non-diagonal entries in $X$ are not equal to 1.\end{proof}

In the next two propositions, we obtain upper and lower bounds for the diagonal Vasquez number.
\begin{prop}{\label{thma(upper)}}
	For each $k\geq 2$, we have $n_d(C_2^k)\leq 5\cdot 2^{k-3}+1$.
\end{prop}
\begin{proof}
	We proceed by induction on $k$. First, consider the base case where $k=2$. Let $\Gamma$ be an $n$-dimensional Bieberbach group of diagonal type with holonomy $C_2^2$ and $n\geq 4$. We claim that there exists a $C_2^2$-homomorphism $f:\Z^n\to\Z^s$ where $s\leq 3$ such that $f^*(\omega)$ is a special element, where $\omega\in H^2(C_2^2;\Z^n)$ is the cohomology class corresponding to standard extension of $\Gamma$. If the claim is true, then $n_d(C_2^2)\leq 3$. 
	
	We proceed by induction on $n$. First, assume $\Gamma$ is  $4$-dimensional. By Proposition \ref{p col-red}, its characteristic matrix $A_\Gamma$ is col-reducible. Thus by Lemma \ref{lem_colred}, there exists a $C_2^2$-homomorphism $f:\Z^4\to\Z^3$ such that $f^*(\omega)$ is a special element. Now, assume the statement is true for $n\leq t-1$ and consider the case where $\Gamma$ is a $t$-dimensional Bieberbach group of diagonal type and $t\geq 5$. By Proposition \ref{p col-red}, its characteristic matrix $A_\Gamma$ is col-reducible. Thus by Lemma \ref{lem_colred}, there exists a $C_2^2$-homomorphism $f:\Z^t\to\Z^{t-1}$ such that $f^*(\omega)$ defines a Bieberbach group. If $f^*(\omega)$ defines a Bieberbach group where its holonomy is $C_2^2$, then by induction hypothesis, there exists a $C_2^2$-homomorphism $g:\Z^{t-1}\to\Z^s$ where $s\leq 3$ such that $g^*(f^*(\omega))$ is a special element. Now, we assume $f^*(\omega)$ defines a Bieberbach group where its holonomy group is a proper subgroup of $C_2^2$. Since  $n_d(C_2)=1$, there must exist a $C_2^2$-homomorphism $g:\Z^t\to \Z$ such that $g^*(f^*(\omega))$ is a special element. Thus the base case where $k=2$ holds.
	
	We assume now that the statement is true for $k\leq q-1$ and consider the case when $k=q\geq 3$. 
Let $\Gamma$ be an $n$-dimensional Bieberbach group of diagonal type with holonomy $C_2^q$ and $n> 5\cdot 2^{q-3}+1$. By a similar induction method as above, we can show that there exists a $C_2^q$-homomorphism $f:\Z^n\to\Z^s$ where $s\leq 5\cdot 2^{q-3}+1$ such that $f^*(\omega)$ defines a Bieberbach group of dimensional at most $5\cdot 2^{q-3}+1$, where $\omega \in H^2(C_2^3;\Z^n)$ is the cohomology class corresponding to the standard extension of $\Gamma$. Thus we have $n_d(C_2^q)\leq 5\cdot 2^{q-3}+1$.
	\end{proof}

\begin{prop}{\label{thma(low)}} Set $a=\frac{k(k-1)}{2}$. For each $k\geq 2$, we have
		\begin{align*}
		n_d(C_2^k) &\geq
		\begin{cases}
		k+a		& \text{if $k$ is even},\\
		k+a-1       & \text{if $k$ is odd}.  
		\end{cases}
		\end{align*}
\end{prop}
\begin{proof}  We will consider the cases where $k$ is even and $k$ is odd separately.
	
{\underline{{\bf Case:}~$k$ is even.}} First, we assume $k$ is even.  It suffices to construct a characteristic matrix $A$ and show that it defines  a $(k+a)$-dimensional Bieberbach group $\G=\G_{\widetilde A}$ of diagonal type with $\omega$ the cohomology class corresponding to standard extension of  $\Gamma$ such that there does not exist a $C_2^k$-homomorphism $f$ with $f^*(\omega)$ defining a smaller dimensional Bieberbach group.  
	
	To this end, define a $k\times k$-matrix $Q$ such that
	\begin{align*}
	Q_{ij}=
	\begin{cases}
	1 & \text{if $i=j$},\\
	2 & \text{if $i\neq j$},
	\end{cases}
	\end{align*}
	for $1\leq i\leq k$ and $1\leq j\leq k$. Let $S=\{(x,y)\in\{1,...,k\}\times\{1,...,k\} \;|\; x<y\}$. It is easy to see that $|S|=a$.  Enumarate the elements $S$ by $s_j=(s_j^{(1)},s_j^{(2)})$ for $1\leq j\leq a$ using the lexicographic order on the pairs. Define a $k\times a$-matrix $N$ such that
	\begin{align*}
	N_{ij}=
	\begin{cases}
	2 & \text{if $i=s_j^{(1)}$},\\
	3 & \text{if $i=s_j^{(2)}$},\\
	0 & \text{otherwise}.
	\end{cases}
	\end{align*}
	where $1\leq i\leq k$ and $1\leq j\leq a$. In other words, fix $j\in\{1,...,a\}$ and consider the $j^{th}$ column of $N$. The $(s_j^{(1)})^{th}$ entry of the $j^{th}$ column of $N$ is equal to 2, the $(s_j^{(2)})^{th}$ entry of the $j^{th}$ column of $N$ is equal to 3 and all other entries of the $j^{th}$ column of $N$ is equal to 0. 
	
	Define a $k\times (k+a)$-matrix $A=\begin{pmatrix}
	Q & N
	\end{pmatrix}$ by combining $Q$ and $N$ together. Let $g_1,...,g_k$ be generators of $C_2^k$ and $M_z\cong\Z$ for all $1\leq z\leq k+a$. By Section \ref{sec_diag},  the matrix $A$ defines a $(2^k-1)\times (k+a)$-matrix $\widetilde{A}$ and a $(k+a)$-dimensional crystallographic group $\Gamma=\Gamma_{\widetilde{A}}$ of diagonal type so that $A_\Gamma(g_i,M_j)=\widetilde{A}_{i,j}$ for $1\leq i\leq k$ and $1\leq j\leq k+a$.

	For example, if we assume $k=2$, then $S=\{(1,2)\}$. We define $Q$, $N$ and $A$ as below,
	$$Q=\begin{pmatrix}1 & 2\\ 2 & 1\end{pmatrix}\,\,  ,  \,\,N=\begin{pmatrix}2\\3\end{pmatrix}\,\,\,\,\text{and}\,\,\,\, A= \begin{pmatrix}
	1 & 2 & 2 \\
	2 & 1 & 3 
	\end{pmatrix}$$
	and the third row of $A_\Gamma$ can be calculated by Lemma \ref{lem_sum}. We get
	$$\bordermatrix{& & & \cr
                r_1				& 1 & 2 & 2 \cr
                r_2				& 2 & 1 & 3 \cr
                r_1\star r_2		& 3 & 3 & 1 \cr}=A_\Gamma$$

	Going back to the general case, we denote the $i^{th}$ row of $A$ to be $r_i$. 
	
	Next, we are going to show that  $\Gamma$ is in fact a Bieberbach group with holonomy group $C_2^k$, by using Lemma \ref{lem_row_A}. Let $r$ be an arbitrary row of $A_\Gamma$. There exists $m\in\{1,...,k\}$ and $1\leq i_1<...<i_m\leq k$ such that the row can be expressed as $r=r_{i_1}\star\cdots\star r_{i_m}$. Notice that the $j^{th}$ column of the row $r$ equals to $A_\Gamma(g_{i_1}\cdots g_{i_m},M_j)$ and
	$$A_\Gamma(g_{i_1}\cdots g_{i_m},M_j)=\bigstar_{1\leq z\leq m} A_\Gamma(g_{i_z},M_j).$$
	We claim that there exist $c_1,c_2\in\{1,...,k+a\}$ such that $A_\Gamma(g_{i_1}\cdots g_{i_m},M_{c_1})=1$ and $A_\Gamma(g_{i_1}\cdots g_{i_m},M_{c_2})\in\{2,3\}$. In other words, we claim that in row $r$ there exists an entry equal to 1 and  an entry equal to either 2 or 3. By Lemma \ref{lem_row_A}, we can conclude that $A_\Gamma$ defines a $(k+a)$-dimensional Bieberbach group of diagonal type where its holonomy group is isomorphic to $C_2^k$.
	
	Let us now prove the claim. First, the claim clearly holds for $m=1$. Assume $2\leq m\leq k$. There exists $j\in\{1,...,a\}$ such that $s_j=(i_1,i_2)\in S$. Then we have
	$$A_\Gamma(g_{i_1}\cdots g_{i_m},M_{k+j})=\bigstar_{z=1}^m A_\Gamma(g_{i_z},M_{k+j})=0\star\cdots \star 2\star 0\star\cdots\star 0\star 3\star 0\star\cdots\star 0=1$$
	It remains to show that there exists an entry equal to 2 or 3 in row $r$ where $2\leq m\leq k$. 
	If $m$ is odd, then pick $i\in\{1,...,k\}\smallsetminus\{i_1,...,i_m\}$. Thus we have
	$$A_\Gamma(g_{i_1}\cdots g_{i_m},M_{i})=\bigstar_{z=1}^m A_\Gamma(g_{i_z},M_{i})=\overbrace{2\star\cdots \star 2}^{\text{odd copies}}=2$$
	If $m$ is even, then we have
	$$A_\Gamma(g_{i_1}\cdots g_{i_m},M_{i_1})=\bigstar_{z=1}^m A_\Gamma(g_{i_z},M_{i_1})=\overbrace{2\star\cdots \star 2\star 1\star 2\star\cdots\star 2}^{\text{odd copies of 2}}=3$$

	This proves the claim. 
	
	By Lemma \ref{lem_row_A}, the matrix $A_\Gamma$ defines a $(k+a)$-dimensional Bieberbach group of diagonal type with holonomy group isomorphic to $C_2^k$. 
	Next, notice that $A_\Gamma$ is equivalent to
	$$\begin{pmatrix}
	\begin{matrix}
	r_1\\
	\vdots\\
	r_k\\
	r_{s_1^{(1)}}\star r_{s_1^{(2)}}\\
	\vdots\\
	r_{s_{a}^{(1)}}\star r_{s_{a}^{(2)}}\\
	\end{matrix}\\
	P
	\end{pmatrix}=\begin{pmatrix}
	X\\
	P
	\end{pmatrix}$$
	where $X$ is a $(k+a)\times (k+a)$-matrix such that the diagonal entries all equals to 1 and all other entries are not equal to 1. By Proposition \ref{prop_col_irred}, we can conclude that $A_\Gamma$ is col-irreducible. Let $\rho_i:C_2^k\rightarrow \mathrm{GL}(M_i)$ be the representation given by the $C_2^k$-action on $M_i$ for each $1\leq i \leq k+a$. Observe that the columns of $\Phi(A_\Gamma)$ are all distinct. So, by Lemma \ref{lem_ker}, we have $ker(\rho_i)\neq ker(\rho_j)$ for all $i\neq j$. By Corollary \ref{r col-irred and ker of rep}, there does not exist an $C_2^k$-homomorphism $f$ such that $f^*(\omega)$ defines a smaller dimensional Bieberbach group where $\alpha$ is the cohomology class defining $\Gamma$. Hence we have $n_d(C_2^k)\geq k+a$ if $k$ is even.
	
	{\underline{{\bf Case:}~$k$ is odd.}} Now, we assume $k$ is odd.  We are going to construct a matrix $A$ and show that it defines a $(k+a-1)$-dimensional Bieberbach group  $\Gamma=\Gamma_{\widetilde A}$ of diagonal type such that there does not exist a $C_2^k$-homomorphism $f$ such that $f^*(\omega)$ defines a smaller dimensional Bieberbach group where $\omega$ is the cohomology class defining $\Gamma$.  First, we assume $k\geq 5$. We will deal with the case $k=3$ afterwards.
	
Define a $k\times k$-matrix $Q$ where
	\begin{align*}
	Q_{ij}=
	\begin{cases}
	1 & \text{if $i=j$},\\
	2 & \text{if $i\neq j$},
	\end{cases}
	\end{align*}
	for $1\leq i\leq k$ and $1\leq j\leq k$. Let $S=\{(x,y)\in\{1,...,k\}\times\{1,...,k\}|x<y\}\smallsetminus \{(1,2),(1,3)\}$. It is easy to see that $|S|=a-2$. Enumerate the elements of $S$ by $s_j=(s_j^{(1)},s_j^{(2)})$ for $1\leq j\leq a-2$ using the lexicographic order on the pairs. Define a $(k\times (a-2))$-matrix $N$ where
	\begin{align*}
	N_{ij}=
	\begin{cases}
	2 & \text{if $i=s_j^{(1)}$}\\
	3 & \text{if $i=s_j^{(2)}$}\\
	0 & \text{otherwise}
	\end{cases}
	\end{align*}
	Define a $k\times (k+a-1)$ matrix $A$ such that
	$$A=\small \begin{pmatrix}
	Q & N & \begin{matrix}
	2 \\ 3\\ 3\\ 0\\ \vdots \\ 0\\
	\end{matrix}
	\end{pmatrix}.$$
	Let $g_1,...,g_k$ be generators of $C_2^k$ and $M_z\cong\Z$ for $1\leq z\leq k+a-1$.  By Section \ref{sec_diag},  the matrix $A$ defines a $(2^k-1)\times (k+a-1)$-matrix $\widetilde{A}$ and a $(k+a-1)$-dimensional crystallographic group $\Gamma=\Gamma_{\widetilde{A}}$ of diagonal type so that $A_\Gamma(g_i,M_j)=\widetilde{A}_{i,j}$ for $1\leq i\leq k$ and $1\leq j\leq k+a-1$.

	For example, if $k=5$, we have
	\[
	\setcounter{MaxMatrixCols}{14}
	A=\begin{pmatrix}
	1&2&2&2&2&2&2&0&0&0&0&0&0&2\\
	2&1&2&2&2&0&0&2&2&2&0&0&0&3\\
	2&2&1&2&2&0&0&3&0&0&2&2&0&3\\
	2&2&2&1&2&3&0&0&3&0&3&0&2&0\\
	2&2&2&2&1&0&3&0&0&3&0&3&3&0
	\end{pmatrix}
	\]

	Going back to the general case, we denote the $i^{th}$ row of $A$ to be $r_i$. 
	
	Next, we are going to show that $\Gamma$ is a Bieberbach group with holonomy group $C^k_2$, by using Lemma \ref{lem_row_A}.  Let $r$ be an arbitrary row of $A_\Gamma$. There exists $m\in\{1,...,k\}$ and $1\leq i_1<...<i_m\leq k$ such that the row can be expressed as $r=r_{i_1}\star\cdots\star r_{i_m}$. Notice that the $j^{th}$ column of the row $r$ equals to $A_\Gamma(g_{i_1}\cdots g_{i_m},M_j)$ and
	$$A_\Gamma(g_{i_1}\cdots g_{i_m},M_j)=\bigstar_{1\leq z\leq m} A_\Gamma(g_{i_z},M_j)$$
	We claim that there exists $c_1,c_2\in\{1,...,k+a-1\}$ such that $A_\Gamma(g_{i_1}\cdots g_{i_m},M_{c_1})=1$ and $A_\Gamma(g_{i_1}\cdots g_{i_m},M_{c_2})\in\{2,3\}$. In other words, we claim that on row $r$ there is an entry equal to 1 and an entry equal to either 2 or 3. By Lemma \ref{lem_row_A}, we can conclude that $A_\Gamma$ defines a $(k+a-1)$-dimensional Bieberbach group of diagonal type where its holonomy group is isomorphic to $C_2^k$.
	
	 First, it is clear that the claim is true for $m=1$.

	Next, we assume  $2\leq m\leq k-1$. If $(i_{m-1},i_m)\in\{(1,2),(1,3)\}$, then $m=2$ and we have
	$$A_\Gamma(g_{i_1}g_{i_2},M_{k+a-1})=A_\Gamma(g_{i_1},M_{k+a-1})\star A_\Gamma(g_{i_2},M_{k+a-1})=2\star 3=1$$
	If $(i_{m-1},i_m)\not\in\{(1,2),(1,3)\}$, then there exists $j\in\{1,...,a-2\}$ such that $s_j=(i_{m-1},i_m)\in S$. Then we have
	$$A_\Gamma(g_{i_1}\cdots g_{i_m},M_{k+j})=\bigstar_{z=1}^m A_\Gamma(g_{i_z},M_{k+j})=0\star\cdots \star 2\star 0\star\cdots\star 0\star 3\star 0\star\cdots\star 0=1$$
	Next, we are going to show there exists an entry equal to 2 or 3 in row $r$. If $m$ is even, we have
	$$A_\Gamma(g_{i_1}\cdots g_{i_m},M_{i_1})=\bigstar_{z=1}^k A_\Gamma(g_{i_z},M_{i_1})=\overbrace{2\star\cdots\star 2\star 1\star 2\star\cdots\star 2}^{\text{odd copies of 2}}=3$$ 
	
	If $m=k$, then 
	$$A_\Gamma(g_1\cdots g_k,M_{k+a-1})=\bigstar_{z=1}^k A_\Gamma(g_z,M_{k+a-1})=2\star 3 \star 3\star 0\star \cdots\star 0=2$$
	If $m$ is odd and $m\neq k$, then pick $i\in\{1,...,k\}-\{i_1,...,i_m\}$ and we have
	$$A_\Gamma(g_{i_1}\cdots g_{i_m},M_{i})=\bigstar_{z=1}^k A_\Gamma(g_{i_z},M_{i})=\overbrace{2\star\cdots\star 2}^{\text{odd copies}}=2$$

	This finishes the claim. By Lemma \ref{lem_row_A}, $A_\Gamma$ defines a $(k+a-1)$-dimensional Bieberbach group of diagonal type $\Gamma$ where its holonomy group isomorphic to $C_2^k$.

	Next, notice that $A_\Gamma$ is equivalent to
	$$\begin{pmatrix}
	\begin{matrix}
	r_1\\
	\vdots\\
	r_k\\
	r_{s_1'^{(1)}}\star r_{s_1'^{(2)}}\\
	\vdots\\
	r_{s_{a-2}'^{(1)}}\star r_{s_{a-2}'^{(2)}}\\
	r_1\star r_2\\
	\end{matrix}\\
	P
	\end{pmatrix}=\begin{pmatrix}
	X\\
	P
	\end{pmatrix}$$
	where $X$ is a $((k+a-1)\times (k+a-1))$ matrix such that the diagonal entries all equal to 1 and all other entries are not equal to 1. By Proposition \ref{prop_col_irred}, we can conclude that $A_\Gamma$ is col-irreducible. Let $\rho_i:C_2^k\rightarrow \mathrm{GL}(M_i)$ be the representation given by the $C_2^k$-action on $M_i$ for all $1\leq i \leq k+a-1$. Observe that columns of $\Phi(A_\Gamma)$ are all distinct. 
	
	For the case $k=3$, we instead consider the $5$-dimensional Bieberbach group $\Gamma$ and its characteristic  matrix $A_{\Gamma}$ defined in Lemma \ref{lem_example} below. $A_{\Gamma}$ is col-irreducible and the columns of $\Phi(A_\Gamma)$ are all distinct. 
	
	 So, for each odd $k\geq 3$, by Lemma \ref{lem_ker}, we have $\rho_i\neq\rho_j$ for all $i\neq j$. By Corollary \ref{r col-irred and ker of rep}, there does not exist an $C_2^k$-homomorphism $f$ such that $f^*(\omega)$ defines a smaller dimensional Bieberbach group where $\alpha$ is the cohomology class defining $\Gamma$. Hence we have $n_d(C_2^k)\geq k+a-1$ if $k$ is odd.
\end{proof}

\begin{lem}{\label{lem_example}}
There exists a 5-dimensional Bieberbach group of diagonal type $\Gamma$ where its holonomy is $C_2^3$ such that its characteristic matrix $A_\Gamma$ is col-irreducible and the columns of $\Phi(A_\Gamma)$ are distinct.
\end{lem}
\begin{proof}
Let $\Gamma$ be the Bieberbach group enumerated in CARAT as "min.72.1.1.502". Its non-lattice generators are
\[\begin{pmatrix}
1&0&0&0&0&0\\
0&-1&0&0&0&1/2\\
0&0&-1&0&0&0\\
0&0&0&1&0&1/2\\
0&0&0&0&-1&0\\
0&0&0&0&0&1
\end{pmatrix}, \,\,
\begin{pmatrix}
-1&0&0&0&0&0\\
0&-1&0&0&0&0\\
0&0&1&0&0&1/2\\
0&0&0&1&0&1/2\\
0&0&0&0&1&1/2\\
0&0&0&0&0&1
\end{pmatrix}\,\,\text{and}\,\,\,
\begin{pmatrix}
1&0&0&0&0&1/2\\
0&1&0&0&0&1/2\\
0&0&1&0&0&0\\
0&0&0&-1&0&0\\
0&0&0&0&-1&0\\
0&0&0&0&0&1
\end{pmatrix}.
\]
By easy calculation, we have
\[
A_\Gamma=\begin{pmatrix}
0&3&2&1&2\\
2&2&1&1&1\\
1&1&0&2&2\\
2&1&3&0&3\\
1&2&2&3&0\\
3&3&1&3&3\\
3&0&3&2&1
\end{pmatrix}, \,\,\,\, \,\,\,\,
\Phi(A_\Gamma)=
\begin{pmatrix}
p&q&q&p&q\\
q&q&p&p&p\\
p&p&p&q&q\\
q&p&q&p&q\\
p&q&q&q&p\\
q&q&p&q&q\\
q&p&q&q&p
\end{pmatrix}.
\]
By observing the matrix $A_\Gamma$ and $\Phi(A_\Gamma)$, we notice that $A_\Gamma$ is col-irreducible and all columns of $\Phi(A_\Gamma)$ are distinct.
\end{proof}
By combining Proposition \ref{thma(upper)} and Proposition \ref{thma(low)}, we get Theorem \ref{main_bdd}.

\begin{proof}[Proof of Theorem \ref{main_exact}] The case $k=1$ is clear. Assume $k=2$. By Theorem \ref{main_bdd}, we have $3\leq n_d(C_2^2)\leq \frac{7}{2}$. Thus we have $n_d(C_2^2)=3$.

Next, assume $k=3$. By Theorem \ref{main_bdd}, we have $5\leq n_d(C_2^3)\leq 6$. It remains to show that if $\Gamma$ is a $6$-dimensional Bieberbach group of diagonal type with holonomy $C_2^3$ and  $\omega\in H^2(C_2^3;\Z^6)$ is the cohomology class corresponding to standard extension of $\Gamma$, then there exists $f:\Z^6\rightarrow\Z^5$ such that $f^*(\omega)$ is special. 

Assume by contradiction that there does not exist such $f$ and hence we assume $A_\Gamma$ is col-irreducible. By Proposition \ref{prop_col_irred}, we assume $A_\Gamma=\begin{pmatrix}
	X\\N
	\end{pmatrix}$ 
	where the entries of $X$ equal to $1$ are exactly on the diagonal and $N$ is a row matrix. By Remark \ref{remk_col=1}, each column of $A_\Gamma$ has at least two entries equal 1. This forces $N$ to have all entries equal to 1. By Lemma \ref{lem_row_A}, the holonomy group of $\Gamma$ is not isomorphic to $C_2^3$, which is a contradiction. We conclude that  $n_d(C_2^3)=5$.
	
	Now we assume $k=4$. By Theorem \ref{main_bdd}, we have $10\leq n_d(C_2^4)\leq 11$. It remains to show that if $\Gamma$ is a $11$-dimensional Bieberbach group of diagonal type with holonomy $C_2^4$ and $\omega\in H^2(C_2^4,M_1\oplus\cdots\oplus M_{11})$ where $M_j\cong \Z$ for $j=1,...,11$ is the cohomology class corresponding to standard extension of $\Gamma$, then there exists a $C_2^4$-homomorphism $f:\Z^{11}\rightarrow\Z^{10}$ such that $f^*(\omega)$ is special. 
	
Assume by contradiction that there does not exist such $f$. Then $A_\Gamma$ is col-irreducible. By Proposition \ref{prop_col_irred}, we assume $A_\Gamma=\begin{pmatrix}
	X\\N
	\end{pmatrix}$ where the entries of $X$ equal to 1 are exactly on the diagonal and $N$ is matrix with four rows. By Remark \ref{remk_col=1}, each column of $A_{\Gamma}$ has either 4 or 8 entries equal to 1. This implies that each column of $N$ has exactly three entries equal to 1. Since $N$ has eleven columns, there exist $i,j\in\{1,...,11\}$ such that 
	$$|\{g\in C_2^4 \;|\; A_{\Gamma}(g,M_i)=A_{\Gamma}(g,M_j)=1\}|= 3.$$
	By Corollary \ref{cor_col=col}, we have
	$$\{g\in C_2^4\; | \; A_{\Gamma}(g,M_i)=1\}=\{g\in C_2^4 \;|\; A_{\Gamma}(g,M_j)=1\}$$
	It follows that $A_{\Gamma}$ is col-reducible, which is a contradiction. Hence $n_d(C_2^4)=10$.
	\end{proof}	

\section{Diffuseness and Bieberbach groups of diagonal type}{\label{sec_diffuse}}

In this section, we use the diagonal Vansquez invariant to characterise non-diffuse Bieberbach groups of diagonal type. First, we state some known results.

\begin{defn}
	{\normalfont Let $G$ be a group and $A\subseteq G$ be a subset. An element $a\in A$ is an {\it extremal point} of $A$ if for all $g\in G/\{1\}$, either $ga\not\in A$ or $g^{-1}a\not\in A$. Define
	$$\Delta(A) =\{a\in A\,|\,a\,\, \text{is an extremal point}\}$$
	The group $G$ is said to be {\it diffuse} if all finite subsets $A\subseteq G$ with $|A|\geq 2$ have $|\Delta(A)|\geq 2$. The group $G$ is said to be {\it weakly diffuse} if all finite non-empty subsets $A\subseteq G$ have $|\Delta(A)|\geq 1$.}	
\end{defn}

The above definition was introduced by B.~Bowditch in \cite{Bowditch}. By \cite[Proposition 6.2]{LM} of P.~Linnell and D.~W.~Morris, a group is diffuse if and only if it is weakly diffuse. 

\begin{remk}{\label{remk_diff_subgp}}
{\normalfont Let $\Gamma$ be a group and let $N\leq \Gamma$ be a subgroup. Straight from the definition of diffuseness, it follows that if $N$ is non-diffuse, then $\Gamma$ is non-diffuse.}
\end{remk}

\begin{prop}[{\cite[Theorem 1.2(1)]{Bowditch}}]{\label{prop diff ppt}}
	 Let $\Gamma$ be a torsion-free group and $N\unlhd\Gamma$. If $N$ and $\Gamma/N$ are both diffuse, then $\Gamma$ is diffuse.
\end{prop}

Recall that the {\it first Betti number} $b_1(\Gamma)$ of a group $\Gamma$ is defined as the (torsion-free) rank of the abelian group $H_1(\Gamma,\Z)=\Gamma/[\Gamma,\Gamma]$.
Let $\Gamma$ be an $n$-dimensional Bieberbach group with holonomy group $G$. By \cite[Corollary 1.3]{HS}, we have $b_1(\Gamma)=rk((\Z^n)^G)$ where $G$ acts on $\Z^n$ via the holonomy representation. The following lemma is straightforward.

\begin{lem}{\label{lem-betti}} Let $\Gamma$ be a group and $N\unlhd \G$. If $b_1(\Gamma)=0$, then $b_1(\G/N)=0$.
\end{lem}

In \cite{carat}, using the  computer programme CARAT, all Bieberbach groups of dimension at most 6 are computed and their standard presentations are given. Combining \cite{KR} and \cite{LSG}, gives a full list of all non-diffuse Bieberbach groups.

\begin{defn}
	{\normalfont An $n$-dimensional closed flat manifold is called a {\it generalized Hantzsche-Wendt (GHW) manifold} if its holonomy group is isomorphic to $C_2^{n-1}$. An oriented GHW-manifold is called a {\it Hantzsche-Wendt (HW) manifold}. A fundamental group of a GHW-manifold $M$ is called a {\it GWH-group} and a {\it HW-group} if $M$ oriented.}
\end{defn}
\begin{remk}
{\normalfont By \cite[Theorem 3.1]{RS}, any generalized Hantzsche-Wendt group is a Bieberbach group of diagonal type.}
\end{remk}

\begin{examp}
	{\normalfont The Bieberbach group enumerated in CARAT as ``group.32.1.1.194'' is a 4-dimensional diffuse GHW-group. Thus not all GHW-groups are non-diffuse.}
\end{examp}

By \cite{KR}, there is only one non-diffuse Bieberbach group of diagonal type of dimension at most three. We denote it by $\Delta_P$. It has a presentation
$$\Delta_P=\langle x,y\,|\, x^{-1}y^2xy^2=y^{-1}x^2yx^2=1\rangle$$
$\Delta_P$ is the $3$-dimensional Hantzsche-Wendt group (also known as Promislow group or Passman group). 

\begin{prop}{\label{c2xc2}}
	 If $\Gamma$ be an $n$-dimensional Bieberbach group of diagonal type with holonomy $C_2^2$ and $b_1(\Gamma)=0$, then $\Delta_P\leq\Gamma$.
\end{prop}
\begin{proof}
	Since $\Gamma$ is Bieberbach group of diagonal type and $b_1(\Gamma)=0$, without loss of generality, we let
	$$\alpha=(diag(X_1,...,X_n),(x_1,...,x_n))\,\,\,\,
	\text{and}\,\,\,\,
	\beta=(diag(Y_1,...,Y_n),(y_1,...,y_n)),$$
	 $X_i,Y_i\in\{1,-1\}$ and $x_i,y_i\in\{0,\frac{1}{2}\}$ for all $i\in\{1,...,n\}$ be the non-lattice generators of $\Gamma$. There exists $i,j\in\{1,...,n\}$ such that $(X_i,x_i)=(1,\frac{1}{2})$ and $(Y_j,y_j)=(1,\frac{1}{2})$. Otherwise, $\alpha\in\Gamma$ or $\beta\in\Gamma$ is an element of order 2, which contradicts the fact that $\Gamma$ is torsion-free. There exists $k\in\{1,...,n\}$ such that $X_k=Y_k=-1$ and $(x_k,y_k)\in\{(0,\frac{1}{2}),(\frac{1}{2},0)\}$, otherwise $\alpha\beta\in\Gamma$ has order 2. Since $b_1(\Gamma)=0$, there does not exist $z\in \{1,...,n\}$ such that $X_z=Y_z=1$. By the third Bieberbach Theorem, we may assume
	$$\alpha=\begin{pmatrix}
	I_s & 0 & 0 & a_1\\
	0& -I_p& 0 & b_1\\
	0 & 0 & -I_q & c_1\\
	0& 0 & 0 & 1
	\end{pmatrix}\,\,\,\,\,\, \text{and}\,\,\,\,\,\,\, \beta=\begin{pmatrix}
	-I_s & 0 & 0 & a_2\\
	0& I_p& 0 & b_2\\
	0 & 0 & -I_q & c_2\\
	0& 0 & 0 & 1
	\end{pmatrix}$$
	where $s,p,q\in\Z^+$, $a_1,a_2\in\{0,\frac{1}{2}\}^s$, $b_1,b_2\in\{0,\frac{1}{2}\}^p$ and $c_1,c_2\in\{0,\frac{1}{2}\}^q$. In addition, $a_1$, $b_2$, and $c_1-c_2$ are non-zero.
	
By a simple calculation, we checked that $\alpha$ and $\beta$ satisfy the below relation
$$\alpha^{-1}\beta^2\alpha \beta^2=\beta^{-1}\alpha^2\beta\alpha^2=1$$
Since
$$\Delta_P=\langle x,y\,|\, x^{-1}y^2xy^2=y^{-1}x^2yx^2=1\rangle$$
there exists a normal subgroup $N\unlhd \Delta_P$ such that $\langle \alpha,\beta\rangle\cong\Delta_P/N$. Let $\bar{\Gamma}=\langle \alpha,\beta \rangle$. Since $\alpha^2$, $\beta^2$ and $(\alpha \beta)^2$ are three linearly independent elements inside the lattice $\bar{\Gamma}\cap\R^n$ and thus $\dim(\bar{\Gamma})\geq 3$. This implies that $N$ has rank zero and is therefore trivial. Hence we have $\Delta_P\cong\langle \alpha,\beta\rangle \leq\Gamma$.
\end{proof}

 \begin{remk}
{\normalfont The above proposition is a special case of Theorem 1 of  \cite{GS} which also covers the non-diagonal case. We should stress that our proof is different from \cite{GS}.}
\end{remk} 
 
\begin{lem}{\label{d3 quot}}
	Let $\Gamma$ be an $n$-dimensional non-diffuse Bieberbach group of diagonal type with holonomy $C_2^2$. Then there exists $\Z^{n-3}\lhd \Gamma$ such that $\Gamma/\Z^{n-3}\cong \Delta_P$.
\end{lem}
\begin{proof}
	By Theorem \ref{main_exact}, we have $n_d(C_2^2)=3$. So,  there exists $\Z^{s}\lhd \Gamma$ such that $\Gamma/\Z^{s}=\bar{\Gamma}$ with $\dim(\bar{\Gamma})\leq 3$. By Proposition \ref{prop diff ppt}, $\bar{\Gamma}$ is non-diffuse. Since $\Delta_P$ is the only non-diffuse Bieberbach group of dimension at most three, we can conclude that $s=n-3$ and $\bar{\Gamma}\cong \Delta_P$.	
\end{proof}

\begin{proof}[Proof of Theorem \ref{main_c2xc2}]
	Let $\{\alpha,\beta\}$ be a set of non-lattice generators of $\Gamma$. Since $b_1(\Gamma)=k$, without loss of generality, assume 
	$$\alpha=(diag(X_1,...,X_n),(x_1,...,x_n))\,\,\,\,
	\text{and}\,\,\,\,
	\beta=(diag(Y_1,...,Y_n),(y_1,...,y_n))$$
	where $x_j,y_j\in\{0,\frac{1}{2}\}$ for $j\in\{1,...,n\}$, $(X_i,Y_i)\in\{(1,-1),(-1,1),(-1,-1)\}$ for $i\in\{1,...,n-k\}$ and $(X_i,Y_i)=(1,1)$ for $i\in\{n-k+1,...,n\}$.
	
	Given an arbitrary element $\gamma=\left(diag(z_1,...,z_{n-k},1,...,1),(s_1,...,s_n)\right)\in\Gamma$ where $z_i\in\{1,-1\}$ for $i\in\{1,...,n-k\}$ and $(s_1,...,s_n)\in\Q^n$. By \cite{LSG}, there exists a homomorphism $f:\Gamma\rightarrow\Z^k$ which maps $\gamma$ to $(2s_{n-k+1},...,2s_n)\in \Z^k$ and the kernel of the homomorphism is an $(n-k)$-dimensional Bieberbach group. 
	
	We claim that $x_i=0$ for all $i\in\{n-k+1,...,n\}$. Assume by contradiction that there exists $j\in\{n-k+1,...,n\}$ such that $x_j\neq 0$. We have $\alpha\not\in ker(f)$. Then the holonomy group of $ker(f)$ will either be identity or cyclic group of order two. By \cite[Theorem 3.5]{KR}, $ker(f)$ is diffuse. Since $ker(f)$ and $\Z^k$ are both diffuse, by Proposition \ref{prop diff ppt}, $\Gamma$ is diffuse, which is a contradiction. Hence $x_i=0$ for all $i\in\{n-k+1,...,n\}$. By similar argument, we get $y_i=0$ for all $i\in\{n-k+1,...,n\}$. Therefore $\Gamma=Z(\Gamma)\oplus\bar{\Gamma}$, where $Z(\Gamma)$ is the center of $\Gamma$ and $\bar{\Gamma}=ker(f)$. By Lemma \ref{d3 quot}, we have 
	\begin{equation}\label{split}
	\begin{tikzpicture}[node distance=1.5cm, auto]
	\node (GA) {$\bar{\Gamma}$};
	\node (G) [right of=GA] {$\Delta_P$};
	\node (I) [right of=G] {$1$};
	\node (Z) [left of=GA] {$\Z^{n-k-3}$};
	\node (O) [left of=Z] {$0$};
	\draw[->] (GA) to node {$\phi$} (G);
	\draw[->] (G) to node {}(I);
	\draw[->] (O) to node {}(Z);
	\draw[->] (Z) to node {$\iota$}(GA);	 
	\end{tikzpicture}
	\end{equation}
	Notice that $b_1(\bar{\Gamma})=0$, otherwise $b_1(\Gamma)>k$. By Proposition \ref{c2xc2}, we have $\Delta_P\leq\bar{\Gamma}$. By restricting the domain of $\phi$, we have
	$$\begin{tikzpicture}[node distance=1.5cm, auto]
	\node (GA) {$\Delta_P$};
	\node (G) [right of=GA] {$G$};
	\node (I) [right of=G] {$1$};
	\node (Z) [left of=GA] {$H$};
	\node (O) [left of=Z] {$0$};
	\draw[->] (GA) to node {$\phi|_{\Delta_P}$} (G);
	\draw[->] (G) to node {}(I);
	\draw[->] (O) to node {}(Z);
	\draw[->] (Z) to node {$\iota$}(GA);	 
	\end{tikzpicture}$$
	
	where $H$ is a subgroup of $\Z^{n-k-3}$ and $G$ is the image of the map $\phi|_{\Delta_P}$. We claim that $\phi|_{\Delta_P}$ is an isomorphism. Since $H$ is a subgroup of $\Z^{n-k-3}$, $H$ is a diffuse group. By Proposition \ref{prop diff ppt}, $G$ is non-diffuse, otherwise it contradicts that $\Delta_P$ is a non-diffuse group. Beside, $G$ is a quotient of $\Delta_P$. Hence $G$ is a non-diffuse Bieberbach group of dimension less than or equal to three. Thus $G\cong\Delta_P$ and hence $\phi|_{\Delta_P}$ is an isomorphism. Therefore (\ref{split}) is a split short exact sequence.
\end{proof}

Before proving Theorem \ref{main_diffuse}, we need the below two propositions to consider the case when the Bieberbach group has trivial center. 

\begin{prop}{\label{thm-beti0}}
	Let $\Gamma$ be an $n$-dimensional Bieberbach group of diagonal type with holonomy $C_2^k$ and $b_1(\Gamma)=0$. Let $p:\Gamma\rightarrow C_2^k$ be the standard projection. If $n<2^k-1$, then there exists $C_2^{k-1}\leq C_2^k$ such that $p^{-1}(C_2^{k-1})$ is an $n$-dimensional Bieberbach group with holonomy group $C_2^{k-1}$ and $b_1(p^{-1}(C_2^{k-1}))=0$.
\end{prop}

\begin{proof}
First note that there are $2^k-1$ subgroups in $C_2^k$ isomorphic to $C_2^{k-1}$.  Let $A_1,...,A_{2^k-1}$ denote these subgroups. Assume by contradiction that $p^{-1}(A_i)$ is Bieberbach group with non-trivial first Betti number for all $i\in\{1,...,2^k-1\}$. Hence we have $(\Z^n)^{C_2^k}=0$ and $(\Z^n)^{A_i}\neq 0$ for all $i\in\{1,...,2^k-1\}$ where $\Z^n\cong\Gamma\cap\R^n$. Fix $i\in\{1,...,2^k-1\}$, let $z_i\in(\Z^n)^{A_i}$ and $e_1,...,e_n$ be the standard basis of $\Z^n$ such that $C_2^k$ acts diagonally on $\{e_1,...,e_n\}$. We can express $z_i$ in term of the basis elements $z_i=c_1e_1+\cdots +c_ne_n$. For all $g\in A_i$, we have 
$$z_i=g\cdot z_i=c_1(g\cdot e_1)+...+c_n(g\cdot e_n)$$
Thus if $c_j\neq 0$ then $e_j\in (\Z^n)^{A_i}$ where $j\in\{1,...,n\}$. We conclude that for each $i\in\{1,...,2^k-1\}$, there exists $t_i\in\{1,...,n\}$ such that $e_{t_i}\in (\Z^n)^{A_i}$. Notice that $t_i\neq t_j$ for all $i\neq j$, otherwise $e_{t_i}\in(\Z^n)^{A_i}\cap(\Z^n)^{A_j}=(\Z^n)^{C_2^k}$, which contradicts that $b_1(\Gamma)=0$. Thus we have $n\geq 2^k-1$ which is a contradiction.
\end{proof}

Recall that a group said to be {\it poly-$\Z$} if it has a normal series with factor groups isomorphic to $\Z$.

\begin{prop}{\label{thm-for e}}
Let $\Gamma$ be a Bieberbach group of diagonal type with $b_1(\Gamma)=0$. Then there exists $\Gamma'\leq\Gamma$ and  a poly-$\Z$ subgroup $N\unlhd\Gamma'$ such that $\Gamma'/N\cong \Delta_P$.
\end{prop}
\begin{proof}
Let $C_2^p$ where $p\geq 1$ be the holonomy group of $\Gamma$. We proceed by induction on $p$. By Theorem \ref{main_c2xc2}, we know the statement holds for $p=2$. Assume that it is true for all $p\leq k-1$ and consider the case where $p=k$. Let $\Gamma$ be an $n$-dimensional Bieberbach group of diagonal type with holonomy $C_2^k$ and $b_1(\Gamma)=0$. By Theorem \ref{main_bdd}, there exists a free abelian group $\Z^t$ for some $t\geq 0$ such that $\bar{\Gamma}=\Gamma/{\Z^t}$ is a Bieberbach group of diagonal type of dimension at most $n_d(C_2^k)<2^k-1$. By Lemma \ref{lem-betti}, we have $b_1(\bar{\Gamma})=0$. If $hol(\bar{\Gamma})\cong C_2^k$, then by Proposition \ref{thm-beti0}, there exists $\G_1\leq\bar{\Gamma}$ such that $b_1(\G_1)=0$ with $hol(\Gamma_1)\cong C_2^{k-1}$. If $hol(\bar{\Gamma})$ is a proper subgroup of $C_2^k$, then we take $\Gamma_1=\bar{\Gamma}$. In other words, $\Gamma_1$ is a Bieberbach group of diagonal type with $b_1(\Gamma_1)=0$ and its holonomy group is $C_2^s$ where $s\leq k-1$. By induction hypothesis, there exists $\Gamma_2\leq\Gamma_1$ and poly-$\Z$ subgroup $A\unlhd\Gamma_2$ such that $\Gamma_2/A\cong\Delta_P$. Since $A\unlhd\Gamma_2\leq\bar{\Gamma}=\Gamma/\Z^t$, we have $\Gamma_2=\G'/\Z^t$ and $A=N/\Z^t$ where $N\unlhd \G'\leq \Gamma$.  We get $\Delta_P\cong \G'/N$.
\end{proof}

\begin{proof}[Proof of Theorem \ref{main_diffuse}]
Let $\Gamma$ be an $n$-dimensional non-diffuse Bieberbach group of diagonal type. By Theorem \ref{equiv_Bieb}, there exists a nontrivial subgroup $\Gamma'$ of $\Gamma$ such that $b_1(\Gamma')=0$. By Proposition \ref{thm-for e}, there exists $\Gamma''\leq\Gamma'$ and a  poly-$\Z$ subgroup $N\unlhd \Gamma''$ such that $\Gamma''/N\cong\Delta_P$.

Now, assume that $\Gamma$ is a non-diffuse generalized Hantzsche-Wendt group. By \cite[Theorem 3.1]{RS}, $\Gamma$ is a Bieberbach group of diagonal type. Let the holonomy group of $\Gamma$ be $C_2^p$. We proceed by induction on $p$ to show that $\Delta_P\leq \Gamma$. The base case $p=2$ is clear. Assume that the statement is true for all $p\leq k-1$ and consider $p=k$. If $b_1(\Gamma)=0$, then by \cite[Proposition 8.2]{RS}, we have $\Delta_P\leq\Gamma$. Hence we may assume $b_1(\Gamma)>0$. By \cite[Proposition 4.1]{RS}, there exists $f:\Gamma\rightarrow\Z$ such that $ker(f)$ is an $(n-1)$-dimensional generalized Hantzsche-Wendt group. Since $\Gamma$ is non-diffuse, by Proposition \ref{prop diff ppt}, $ker(f)$ is non-diffuse. Hence by induction hypothesis, we have $\Delta_P\leq ker(f)\leq\Gamma$.
\end{proof}


\begin{thebibliography}{12345}

	\bibitem{Bowditch}
	B.H. Bowditch, A variation on the unique product property, J. London Math. Soc. (2), 62(3):813-826,(2000)
	
	\bibitem{Brown}
	K.S, Brown, Cohomology of groups, Springer, New York, 1982
	
	\bibitem{carat}
	CARAT, Crystallographic Algorithms And Tables. (2006). Lehrstuhl B fur Mathematik, Aachen. Available at: http://wwwb.math.rwth-aachen.de/carat/
	
	\bibitem{charlap}
	L. S. Charlap, Bieberbach Groups and Flat Manifolds, Universitext, SpringerVerlag, Berlin, (1986)
	
	\bibitem{CW}
	G. Cliff, A. Weiss, Torsion Free Space Groups and Permutation Lattices for Finite Groups, Contemp. Math. $\boldsymbol{93}$ (1989), 123-132
	
	\bibitem{CR} 
	J.H. Conway, J.P. Rossetti, Describing the Platycosms, Preprint, http://arxiv.org/abs/math.DG/0311476
	
	\bibitem{DW}
	M. Dadarlat, E. Weld, Connective Bieberbach groups, International journal of Math. Vol. 31, No. 06, 2050047 (2020)
	
	\bibitem{Gardam}
	G. Gardam, A counterexample to the unit conjecture for group rings, Ann. Math., to appear, Math ArXiv:2102.11818 (2021)

	
	\bibitem{GS}
	N. Gupta, S. Sidki, On torsion-free metabelian groups with commutator quotients of prime exponent, International journal of algebra and computation, $\boldsymbol{9}$ (1999), 493-520

\bibitem{higman}
	G. Higman, Units in group rings. D.Phil. thesis, University of Oxford, (1940)

	\bibitem{HS}
	H. Hiller, C. H. Sah, Holonomy of flat manifolds with $b_1$=0, Quarterly J. Math. $\boldsymbol{37}$ (1986), 177-187
	
	\bibitem{KR}
	S. Kionke, J. Raimbault, On geometric aspects of diffuse groups, Doc. Math., 21:873-915, (2016) With an appendix by Nathan Dunfield.
	
	\bibitem{LM}
	P. Linnell, D. W. Morris, Amenable groups with a locally invariant order are locally indicable, Groups, Geometry, and Dynamics $\boldsymbol{8}$ (2014), 467-478
		
	\bibitem{LSG}
	R. Lutowski, A. Szczepanski, A. Gasior, A short note about diffuse Bieberbach groups, Journal of Algebra, $\boldsymbol{494}$ (2018), 237-245
	
	\bibitem{Nan-LPPS}
	R. Lutowski, N. Petrosyan, J. Popko, A. Szczepanski, Spin structures of flat manifolds of diagonal type, 2016, Homology, Homotopy and Application, Volume 21 No. 2 p.333-344.
	
	\bibitem{Lut}
	R. Lutowski. Finite outer automorphism groups of crystallographic groups. Exp. Math. 22(2013), no. 4, 456-464.
	
	\bibitem{MP}
	R.J. Miatello, R.A. Podest\'{a}, The Spectrum of Twisted Dirac Operators on Compact Flat Manifolds, Transactions of the American Mathematical Society, Vol. 358, No. 10 (Oct., 2006), pp. 4569-4603
	
	\bibitem{Pra}
	A. Prasad, Counting subspaces of a finite vector space, Resonance, vol. 15; no. 11, pages 977-987 and no. 12, pages 1074-1083, 2010
	
	\bibitem{PS}
	J.Popko, A.Szczepański, Cohomological rigidity of oriented Hantzsche–Wendt manifolds, Advances in Mathematics
Volume 302, 22 October 2016, Pages 1044-1068

	\bibitem{RS}
	J. P. Rossetti, A. Szczepanski, Generalized Hantzsche-Wendt flat manifolds - Rev. Mat. Iveroamericana $\boldsymbol{21}$ (2005), no.3, 1053-1070	
	\bibitem{Szc cry}
	A. Szczepanski, Geometry of crystallographic groups, World Scientific, 2012
	
	\bibitem{Szc}
	A. Szczepanski, Decomposition of flat manifolds, Mathematika $\boldsymbol{44}$ (1997), 113-119
	
	\bibitem{Prob}
	A. Szczepański, Problems on Bieberbach groups and flat manifolds, Geometriae Dedicata 120 (1), 111-118
	
	\bibitem{Vas}
	A. T. Vasquez, Flat Riemannian manifolds, J. Diff. Geom. $\boldsymbol{4}$ (1970), 367-382

	
\end{thebibliography}
\end{document}